\documentclass{article}
\usepackage[paper=a4paper,left=35mm,right=35mm,top=35mm,bottom=35mm]{geometry}

\usepackage[sort&compress,numbers]{natbib}
\usepackage{placeins}
\usepackage{graphicx}      
\usepackage{natbib}        
\usepackage{amsmath} 
\usepackage{amsthm}
\usepackage{amssymb}  
\usepackage{units} 		
\usepackage{subfigure} 
\usepackage{lineno}		 
\usepackage[affil-it]{authblk}

\newcommand*\diff{\mathop{}\!\mathrm{d}}
\newcommand{\R}{\mathbb{R}}
\newcommand{\T}{\top}
\newcommand{\orthcomp}{\bot}

\usepackage{color}

\newtheorem{corollary}{Corollary}
\newtheorem{lemma}{Lemma}
 
\newtheorem{proposition}{Proposition}

\title{An observer for partially obstructed wood particles in industrial drying processes}

\author[a]{Marc Oliver Berner} 
\author[b]{Viktor Scherer} 
\author[a]{Martin M\"onnigmann\thanks{Corresponding author. E-mail address: martin.moennigmann@rub.de (M. M\"onnigmann).}} 

\affil[a]{Ruhr-Universit\"at Bochum, Automatic Control and Systems Theory, 
Universit\"atsstr. 150, 44780 Bochum, Germany}
\affil[b]{Ruhr-Universit\"at Bochum, Energy Plant Technology,\\ Universit\"atsstr. 150, 44780 Bochum, Germany} 

\date{}
\begin{document}
\maketitle
\vspace{-.5cm}
\begin{abstract}   
In order for biomass drying processes to be efficient, it is crucial to achieve the target residual water content within a close margin, 
since more conservative drying would result in a waste of energy. 
A method for a reliable estimation of the water content is therefore of obvious importance. Ideally, such a method does not require any expensive sensors. 
We show reduced order models and extended Kalman filters can be combined to reliably determine the water content and temperature of wood particles
based on only surface temperature measurements. 
The proposed observer works reliably if measurements are only available for parts of a particle face. 
It can therefore still be applied if particle surfaces are partially obstructed, which is a prerequisite for use in industrial processes and units, such as rotary dryers.
The extended Kalman filter uses a reduced order model that is obtained by applying proper orthogonal decomposition and 
Galerkin projection to coupled PDEs that model heat conduction and water diffusion in anisotropic particles. In contrast to the original PDE simulation model, the reduced model and the filter based on it are suitable for real time computations and monitoring. 
\end{abstract}


\section{Introduction}

Process monitoring and state observation are of obvious importance for ensuring product quality and energy efficiency. 
In the biomass drying processes treated here, the particle moisture must be monitored to determine when the biomass is sufficiently dry for further processing. 
While the particle surface temperature can relatively easily be measured with infrared thermography, the particle moisture cannot be measured directly. 
There exist approaches that use humidity-indicating substances or determine the moisture with balances, but these methods are not practical in an industrial environment \citep{Sudbrock2015}.  

Dynamic models are instrumental to estimate unmeasured quantities in transient processes. 
Modeling the drying processes of the type treated here leads to models in the form of coupled partial differential equations (PDEs) (see, e.g., \cite{Sudbrock2015,RuizLopez2011,DeTemmerman2009}). 
Solving these PDE with finite volume methods is state of the art today. When it comes to model-based observers, however, 
it is often more effective to use model reduction methods first and to apply mature methods for the analysis of ordinary differential equations to the resulting reduced model~\citep{John2010}. 
We apply proper orthogonal decomposition (POD) and Galerkin projection for this purpose~\citep{Moore1981,Jansen2017, Deane1991}. Model reduction approaches based on these steps proved suitable to describe the dynamic behavior of wood chips \citep{Scherer2016,Berner2017} and for the controllability analysis and optimal control of the drying process \citep{Berner2019}. 

In the present paper, we complement the results of \cite{Berner2019} by implementing an observer and considering the observability of the wood chip drying process. Empirical observability Gramians are used for this purpose \citep{Hahn2003, Lall1999}. Furthermore, we establish an extended Kalman filter to monitor the moisture content during the drying process. 
In contrast to the present paper, \cite{Berner2019} addressed optimal control problems and analyzed the controllability properties of the process for this purpose. Both the present paper and~\cite{Berner2019} are based on the same PDE wood chip and reduced order models, which are summarized in Sections~\ref{sec:WoodChipModelling} and~\ref{subsec:modelreduction}, respectively, as required here. 
Some additional references on observability analysis are given in Section~\ref{sec:problemformulation}.

We introduce the wood chip drying process in Section~\ref{sec:WoodChipModelling}. 
The required observability methods and the derivation of the reduced order model (ROM) are summarized in Sections \ref{sec:problemformulation}
and Section~\ref{sec:solutionformulation}, respectively, as required here and applied to the wood chip drying process in 
Section~\ref{sec:ApplicationToWoodChip}.  
A brief conclusion and an outlook are given in Section~\ref{sec:Outlook}.

\section{Wood chip drying model}\label{sec:WoodChipModelling}
A single wood particle can be characterized by its inner-particle heat and moisture distribution, which essentially change due to the evaporation of water at the particle surface. The process must be resolved on the single particle scale to take anisotropic material properties into account \citep{Sudbrock2015,Scherer2016}. A typical wood chip can be assumed to be rectangular and to have dimensions of $\unit[10]{mm}\times\unit[20]{mm}\times\unit[5]{mm}$. 

We summarize the PDE and its boundary and initial conditions in Table~\ref{tb:PDE} and only point out some features of the model, since it has been discussed in detail before~\citep{Berner2019,Scherer2016}.
The transient, inner particle behavior is modeled with Fick's law of diffusion and Fourier's law of heat conduction, which yield~\eqref{eqn:PDEX} and~\eqref{eqn:PDET}, respectively.
The volumetric heat capacity $s(x)$ and the diffusion coefficients $\lambda(x)$ and $\delta(T)$ depend on the local temperature or moisture at spatial location $y$ and time $t$. Note that $\lambda(x(y, t))\in\R^{3 \times 3}$ and $\delta(T(y, t))\in\R^{3 \times 3}$ due to the anisotropy of the wood. 
Heat flux and mass flux due to evaporation at the particle surface $\partial\Omega$ are modeled with Neumann boundary conditions~\eqref{eqn:PDEBC}, which depend on the inner-particle temperature and moisture distribution, the ambient humidity $\varrho_\infty$ and the ambient temperature $T_\infty$.
Mass flux is induced by mass transfer $J_{x}(x,T,\varrho_\infty)$ to the ambient air. 
Heat flux results from heat transfer $J_{T}(x,T,T_\infty)$ from the hot ambient air and accounts for the heat $J_{T,\mathrm{v}}(x,T,\varrho_\infty)$ required for evaporation (see \cite{Berner2019} Sect.~2 and App.~A for details).
\begin{table}[t]\footnotesize
\caption{PDE, boundary and initial conditions for moisture $x(y,t)$ and temperature $T(y, t)$ at location $y$ and time $t$.}\label{tb:PDE}
Partial differential equations: 
\begin{subequations}
\label{eqn:PDE}
\begin{align}
\frac{\partial x}{\partial t} &= \nabla\big( {\delta}(T) \nabla x\big) \label{eqn:PDEX}\\
\frac{\partial T}{\partial t} &= s^{-1}(x) \nabla \big( {\lambda}(x) \nabla T \big)
\label{eqn:PDET}
\end{align}
\end{subequations}
Boundary conditions:\begin{subequations}
\label{eqn:PDEBC}
\begin{align}
n^\T \big(\delta(T) \nabla x \big) \Big \vert_{\partial \Omega} &= J_{x}(x,T,\varrho_\infty) \label{eqn:PDEBCX}\\
n^\T \big( \lambda(x) \nabla T\big) \Big \vert_{\partial \Omega}  &= J_{T}(x,T,T_\infty) + J_{T,\mathrm{v}}(x,T,\varrho_\infty) \label{eqn:PDEBCT}
\end{align}
\end{subequations}
Initial conditions: \begin{subequations}
\label{eqn:PDEInitCond}
\begin{align}
x(y,t=0)&= x_0\;\; \text{for all }y\\
T(y,t=0)&=T_0 \;\; \text{for all }y
\end{align}
\end{subequations}
\end{table}

Solving the PDE with boundary and initial conditions (cf. Table~\ref{tb:PDE}) yields the temperature $T(y,t)$ and moisture $x(y,t)$ at time $t$ and spatial location $y \in \Omega$, where $\Omega \subset \R^3$ is the wood chip domain. 
The overall moisture 
\begin{align}\label{eqn:TotalMoisture}
X(t) = \frac{1}{V}\int_\Omega x(y,t) \diff V,
\end{align}
where $V$ refers to the wood chip volume, serves as a measure for the progress of the drying process. More details on the model \eqref{eqn:PDE}-\eqref{eqn:PDEBC}, the finite volume simulation used to solve \eqref{eqn:PDE}-\eqref{eqn:PDEInitCond} and simulation results can be found in \cite{Sudbrock2014,Sudbrock2015,Scherer2016}. 

We showed in~\cite{Berner2019} that the drying process of a single wood chip can be carried out in an energy-optimal fashion if the ambient temperature $T_\infty$ is controlled as a function of time. The energy-optimal control input functions for $T_\infty$ are determined in~\cite{Berner2019}, however, under the assumption that the temperature $T(x, t)$ and moisture $x(y, t)$ can be determined for all points $y\in\Omega$ and all times $t$. 
In the present paper, we show that these very strong assumptions, which can only be met in simulations but not in practical processes, can be dropped. 


\section{Problem formulation}\label{sec:problemformulation}

The surface temperature distribution of particles in dryers can be measured with infrared thermography~\citep{Rickelt2013}. 
The moisture, in contrast, cannot be measured directly. In \cite{Moreton2002} for example, the surface particle moisture is determined by coating the particle with a humidity-indicating substance that changes color. 
In laboratory experiments, precision balances are used to determine the moisture of the entire bulk gravimetrically. Alternatively, multiple sensors (humidity, temperature and flow rate sensors) in the air in- and outlet of the dryer can be used to determine the moisture in the bulk via the law of mass conservation for the air humidity \citep{Bengtsson2008}. 
All measurement methods mentioned so far incur considerable additional cost. It is therefore of interest to find a model-based method that requires as few measurements as possible. Since temperature measurements are in principle simpler than moisture measurements, any method that avoids the latter is to be preferred. 

We reconstruct the transient moisture and temperature distribution inside a wood chip with an extended Kalman filter for a reduced model. The reduced model results from~\eqref{eqn:PDE} with proper orthogonal decomposition and Galerkin projection. Because the extended Kalman filter is based on a reduced model, it can provide the desired inner-particle moisture and temperature distribution in real-time at a small computational cost.  
The filter only requires particle tracking and measurements of the particle surface temperature, which can be achieved by coupling particle tracking velocimetry and infrared thermography~\citep{Tsuji2010}, for example.
The temperature on the particle surface is not required for all points on the particle surface, but measurements on part of a single face of the particle suffice. Consequently, it can still be applied if part of the particle is covered by neighboring particles.  
We note Kalman filters have been used for state estimation of other distributed parameter systems with a complexity similar to the one treated here before (see, e.g., \cite{John2010, Hoepffner2005}). 

As a preparation, we show that the moisture and temperature distribution inside a single wood chip, i.e., the solution to the PDEs~\eqref{eqn:PDE}  
subject to boundary conditions \eqref{eqn:PDEBC},
can be reconstructed from only a partial measurement of the particle-surface temperatures. This preparatory step corresponds to an observability analysis. 
Various methods are available to analyze the observability of nonlinear distributed parameter systems such as \eqref{eqn:PDE}-\eqref{eqn:PDEInitCond}. Some of them directly consider the PDEs with semigroup theory methods \citep{Zuazua2007,Delrattre2004}. Other approaches analyze the observability of the finite-dimensional approximation that results for spatial discretization \citep{Leon2002, Hahn2003}. Mature observability criteria exist for linear finite-dimensional systems \citep{Chen1999}. However, the discretization required for their application to the drying process results in large systems; a few thousand state variables are used for a single particle in the example of Section~\ref{sec:ApplicationToWoodChip}. 
We apply a model reduction to obtain a finite-dimensional system with an appropriate precision but much smaller size than the original spatially discretized model. 

Before turning to the model reduction in Section~\ref{sec:solutionformulation}, 
the required background on observability analysis methods is summarized in Section~\ref{subsec:CCM}.
Since a large range of ambient conditions needs to be covered, a linearized model and a linear observability analysis are not appropriate. We carry out a nonlinear observability analysis based on empirical Gramians~\citep{Lall1999, Hahn2003}. 

\subsection{Empirical observability Gramian}\label{subsec:CCM}
The observability analysis requires simulation results for \eqref{eqn:PDE}-\eqref{eqn:PDEInitCond}. 
We derive a finite volume model of \eqref{eqn:PDE} for this purpose as follows. The wood chip domain $\Omega$ is tessellated using a Cartesian grid of $N$ cubic finite volumes $\Delta V$, where the $i$th finite volume belongs to location $y_i\in\Omega$, $i=1,\ldots N$. 
The moisture $x(y, t)$ and temperature $T(y, t)$ that result for~\eqref{eqn:PDE} are approximated by the moisture and temperature of the respective finite volume, which we denote $x(y_i,t)$ and $T(y_i,t)$. Gradients $\nabla x(y_i,t)$ and $\nabla T(y_i,t)$ are determined by balancing heat and mass fluxes through each finite volume $\Delta V$. Heat and mass fluxes at the particle surface are replaced by the respective boundary conditions, i.e., by substituting $x(y_i,t)$ and $T(y_i,t)$ in \eqref{eqn:PDEBC}. 
The measured surface temperatures in $\delta\Omega_w$ can then be expressed by the temperatures $T(y_i, t)$ of the surface volume elements. 
We refer to \cite{Moukalled2015, Eymard2000, Fletcher1984} for more details on the finite volume method. 

We collect all $x(y_i,t)$, $T(y_i,t)$ for $i=1,\ldots, N$ in the state vector
\begin{align}
z(t)&=[x(y_1,t) \hdots  x(y_N,t)  \,T(y_1,t)  \hdots  T(y_N,t)]^\T \label{eqn:StateVectorCFD}
\end{align}
for convenience, where $z(t)\in\R^{M}$, with $M=2N$. This yields the discretized model 
\begin{subequations}\label{eqn:NonlinSysAllg}
\begin{align}
\dot{z}(t)&=f\big(z(t)\big)\label{eqn:NonlinSysStateEq}\\
w(t)&=h\big(z(t)\big)\label{eqn:MeasuredOutput},
\end{align}
\end{subequations}
where the function $h:\R^M\rightarrow\R^v$ assigns the particle surface temperatures 
to the output $w(t)\in\R^v$, $v\ll M$. 
The functions $h$ and $w$ will be used to model the measurement at a single point on the particle surface ($v= 1$)
and to model the measurement on a partial particle face in Section~\ref{subsec:OBSVDryingProcess}.

The observability check for \eqref{eqn:NonlinSysAllg} is carried out as follows~\citep{Hahn2003}. Assume $z_\mathrm{ss}$ is a steady state, $f\big(z_\mathrm{ss}\big) = 0$.
We record the output $w_{dli}(t)=h\big(z_{dli}(t)\big)$ that results for the step input
\begin{align}\label{eqn:GramInit}
z_0=h_d D_l e_i +  z_\mathrm{ss},
\end{align}
where $h_d\in\R^+$ are positive constants, $D_l\in\R^{M \times M}$ are orthonormal matrices, 
and $e_i\in\R^M$, $i=1,\ldots,M$ refers to the $i$th standard unit vector. 
We can then determine the empirical observability Gramian
\begin{align}
\label{eqn:Gram}
G_o=\sum_{l=1}^r \sum_{d=1}^s\frac{1}{rsh_d^2}\int_0^\infty   D_l \Psi_{dl} D_l^\T \diff t,
\end{align}
where $G_o \in \R ^{M \times M}$ is symmetric by construction.
The entry $i,j$ of the covariance matrix $\Psi_{dl}$ is defined by its elements
\begin{align}
\label{eqn:GramCov}
\Psi_{dl,ij}=\big(w_{dli}(t)-w_{\mathrm{ss},dli}\big)^\T\big(w_{dlj}(t)-w_{\mathrm{ss},dlj}\big)
\end{align}
 where $w_{\mathrm{ss},dli}= \lim\limits_{t \to \infty} w_{dli}(t)$ accounts for nonzero steady states.

The Gramian~\eqref{eqn:Gram} is called \textit{empirical} observability Gramian, because it depends on the choice of the signals $h_d D_l e_i$ in~\eqref{eqn:GramInit}. To account for nonlinearity, 
\eqref{eqn:Gram} must be calculated for various input magnitudes $h_d$, $d = 1, \dots, s$ 
and perturbation directions $D_l$, $l = 1, \dots, r$. The signals $h_d D_l e_i$ must reflect the signals that occur in the actual process operation~\citep{Hahn2002}. They are selected in Section~\ref{subsec:GramROM} for the particular drying process treated
here.
The empirical observability Gramian \eqref{eqn:Gram} equals the linear observability Gramian if the system \eqref{eqn:NonlinSysAllg} is linear \cite[Lemma 7]{Lall1999}.
We refer to \cite{Moore1981} and \cite{Hahn2003} for details on the observability analysis with empirical Gramians.

Note that a total of $s\cdot r\cdot M$ simulations are necessary to determine \eqref{eqn:Gram}. It is the point of the model reduction to replace the large  system dimension $M$ of the finite volume model by a small system dimension of a reduced order model thus reducing the computational effort of the observability analysis in particular. 

A statement on global observability is not available for nonlinear systems in general, but the following Lemmata are valid locally~\citep{Lall1999}. Let $\beta_i$, $i=1,\ldots,M$ refer to the eigenvalues and $v_i$ to the associated eigenvectors of the eigenvalue problem
\begin{align}\label{eqn:EigProblemLarge}
G_o v_i - \beta_i v_i = 0.
\end{align}

\begin{lemma}\label{lma:LinearObsvCond} (see, e.g., \cite[Chapter 6.3]{Chen1999}) Assume the system \eqref{eqn:NonlinSysAllg} to be linear and stable. Then \eqref{eqn:NonlinSysAllg} is observable with respect to output \eqref{eqn:MeasuredOutput} if and only if $\beta_i>0$ for all $i=1,\ldots,M$, i.e., if and only if the linear observability Gramian is positive definite.
\end{lemma}

\begin{lemma}\label{lma:ObsvSubspace} (see, e.g., \cite{Moore1981}) Let $\beta_k$ and $v_k$, $k=1,\ldots,M$ be the eigenvalues and associated eigenvectors of \eqref{eqn:Gram} for a stable linear system. Then all initial points in the state space that result in an output energy $\int_0^\infty y^\T(\tau) y(\tau) \diff\tau \leq 1$ are located within a hyper-ellipsoid with semi axes $\nicefrac{v_k}{\sqrt{\beta_k}} $, $k=1,\ldots,M$.
\end{lemma}

Essentially, the eigenvalues, where we assume $\beta_1 \geq \ldots \geq \beta_M$ without restriction, and their corresponding eigenvectors $v_1, \ldots , v_M $ specify the range and direction of the most observable states. Note that the smallest eigenvalue refers to the least observable state but results in the major semi-axes of the hyper-ellipsoid.

Furthermore, we use the observability measure 
\begin{align}\label{eqn:ObsvMeausure}
\kappa = \mathrm{trace}\big( G_\text{o} \big) = \sum_{i=1}^{M} \beta_i
\end{align}
that measures how much of the state perturbation is transmitted to the output $w(t)$ (see, e.g., \cite{Singh2005}). The measure $\kappa$ is ultimately used to determine appropriate measurement locations for the Kalman filter in Section \ref{sec:EKF}.

It is impractical to determine the observability Gramian with the finite-volume model, 
because this would require to carry out $s\cdot r\cdot M$ simulations, where $M=2000$ applies here (see Section~\ref{sec:ApplicationToWoodChip} and note this value for $M$ is not particularly large). 
The model reduction technique summarized
in the subsequent section permits to reduce the number of states by two
orders of magnitude.\footnote{Calculating \eqref{eqn:Gram} with the reduced order model requires $\unit[1155]{s}$ with a matlab implementation on a desktop PC (Intel i7-6700 CPU, $\unit[3.4]{GHz}$, $\unit[64]{GiB}$ RAM). The corresponding calculation with the finite volume-model was incomplete after one day.}
Specifically, a reduced order model with ten states will
turn out to be sufficient.

\section{Solution formulation}\label{sec:solutionformulation}
We summarize the required aspects of the model reduction method in Section~\ref{subsec:modelreduction} and explain how to use it for the observability analysis in Section~\ref{subsec:GramROM}.

\subsection{POD and Galerkin projection based reduction}\label{subsec:modelreduction}
We briefly introduce the model reduction procedure as required here and refer to \cite{Berner2017, Scherer2016} for more details. POD and Galerkin projection \citep{Sirovich1987} are applied to \eqref{eqn:PDE}--\eqref{eqn:PDEInitCond}. All explanations use \eqref{eqn:PDET} as an example, \eqref{eqn:PDEX} can be treated analogously. In the explanation to follow, we assume the material parameters $s$ and $\lambda$ in \eqref{eqn:PDET} to be constant for simplicity. We stress that the actual model reduction in Sections \ref{sec:ApplicationToWoodChip} and \ref{sec:EKF} are performed with the non-constant quantities.

The model reduction is based on so called snapshots
\begin{align}\label{eqn:SnapshotsT}
z_T(t_j)&=[T(y_1,t_j) \hdots  T(y_N,t_j)]^\T
\end{align}
$z_T(t_j)\in\R^{N}$ that solve or approximately solve \eqref{eqn:PDET} at times $t_j$, $j=1,\ldots,m$ and spatial locations $y_i\in\Omega$, $i=1,\ldots,N$ for boundary conditions \eqref{eqn:PDEBC} and initial conditions \eqref{eqn:PDEInitCond}. Assume $b$ linear independent snapshots exist. We can then find $b$ orthonormal basis vectors  $\phi_{T,k}=[\varphi_{T,k}(y_1) \hdots  \varphi_{T,k}(y_N) ]^\T $, $\phi_{T,k}\in\R^{N}$, $k=1,\ldots,b$, of the snapshot set, also called modes, such that 
\begin{align}\label{eqn:SumLinComb}
T(y_i,t_j)&= \bar{T}(y_i) + \textstyle\sum_{k=1}^{b} c_{T,k}(t_j) \varphi_{T,k}(y_i),
\end{align} 
where 
\begin{align}\label{eqn:SnpMean}
\bar{T}(y_i) = \tfrac{1}{m}\textstyle\sum_{j=1}^{m} T(y_i,t_j), \in\R,
\end{align}  
is the time average and
\begin{align}\label{eqn:TimeCoeff}
c_{T,k}(t_j)&= \textstyle\sum_{i=1}^{N} \big(T(y_i,t_j)- \bar{T}(y_i)\big) \, \varphi_{T,k}(y_i) \Delta V \nonumber\\&=\langle  z_T(t_j)- \bar{z}_T,\, \phi_{T,k} \rangle, \in\R
\end{align} 
are time dependent coefficients, where $\bar{z}_T=[\bar{T}(y_1),\ldots,\bar{T}(y_N)]^\T$. The brackets $\langle \cdot, \, \cdot \rangle$ denote the inner product 
\begin{align}\label{eqn:InnerProductL2Vector}
\langle a, b\rangle &=a^\T b \Delta V
\end{align}
for $a=[a(y_1),\ldots,a(y_N)]^\T$, $b=[b(y_1),\ldots,b(y_N)]^\T$, $a,b\in\R^N$ and the discrete volume $\Delta V \in \R$.

An approximation for \eqref{eqn:SumLinComb} results by truncating at a cut-off value $n_T \ll b$. The left singular vectors that result from applying singular value decomposition to the set of snapshots constitute suitable modes in the sense that 
\begin{align}\label{eqn:SumLinApprox}
T(y_i,t_j)&\approx \bar{T}(y_i) + \textstyle\sum_{k=1}^{n_T} \varphi_{T,k}(y_i) c_{T,k}(t_j)
\end{align}
yields the best approximation for a given $n_T$~\citep{Sirovich1987,Cordier2008a,Cordier2008b}. The cut-off value $n_T$ should be chosen as small as possible, since $n_T$ corresponds to the order of the reduced model. The singular values $\sigma_1 > \ldots > \sigma_b$ provide a measure 
\begin{align}
E(n_T)=\frac{\sum_{i=1}^{n_T} \sigma_i}{\sum_{j=1}^{b} \sigma_j}
\label{eqn:Energy}
\end{align}
for the quality of \eqref{eqn:SumLinApprox}. The measure \eqref{eqn:Energy} is used to evaluate whether $n_T$ is chosen appropriately, where any $n_T$ that results in $E(n_T) \approx 1$ is appropriate \citep[Ch.~3.6]{Cordier2008a}.

The desired reduced model consists of $n_T$ ordinary differential equations 
\begin{align}\label{eqn:PDEProjection2}
\dot{c}_{T,k}(t_j) &\approx \textstyle\sum_{i=1}^{N} \varphi_{T,k}  \bigg( s^{-1} \nabla \cdot  \Big( {\lambda} \nabla \big(\bar{T}(y_i)+ \nonumber\\ &\textstyle\sum_{l=1}^{n_T} \varphi_{T,l}(y_i)c_{T,l}(t_j)\big) \Big) \bigg)  \Delta V,
\end{align}
for $c_{T,k}(t)$, $k=1,\ldots,n_T$, that approximate \eqref{eqn:TimeCoeff} at $t=t_j$. The ODEs are constructed by substituting \eqref{eqn:SumLinApprox} into \eqref{eqn:PDET}, applying the inner product, and exploiting the time independence and orthonormality of the modes,
\begin{align}\label{eqn:OrthoNormality}
\textstyle\sum_{i=1}^{N}  \varphi_{x,l}(y_i) \;  \varphi_{x,k}(y_i)  \Delta V  =\delta_{l,k}
\end{align}
where $\delta_{l,k}$ refers to Kronecker's delta~\citep{Berner2017, Scherer2016}. 

Equations \eqref{eqn:PDEProjection2} are the desired ODEs, i.e., the reduced order model (ROM) for temperature diffusion. We explicitly consider the boundary conditions \eqref{eqn:PDEBC} in the reduced order model with the help of Gauss's theorem \citep{Berner2017} and denote the resulting equations
 \begin{align}
\dot{c}_{T,l}(t) &= f_{\text{ROM,T},l}\big(c_{x,k}(t),c_{T,l}(t)\big). \label{eqn:PDEProjectionT}
\end{align}
The structure of $f_{\text{ROM, T},l}$ is briefly explained in Appendix A.

We repeat the model reduction procedure for the moisture diffusion \eqref{eqn:PDEX} with the moisture approximation
\begin{align}\label{eqn:SumLinApproxWater}
x(y_i,t_j)&\approx \bar{x}(y_i) + \textstyle\sum_{k=1}^{n_x} \varphi_{x,k}(y_i) c_{x,k}(t_j),
\end{align}
where $\bar{x}(y_i)$, $n_x$, $\varphi_{x,k}(y_i) $ and $c_{x,k}(t_j)$, $i=1,\ldots,N$, $j=1,\ldots,m$, $k=1,\ldots,n_x$ are the time average, cut-off value, modes and time coefficients, respectively, that were determined with the methods presented above from snapshots for the moisture. Applying Galerkin projection and Gauss's theorem to \eqref{eqn:PDEX} yields the desired ODEs which we denote
 \begin{align}
\dot{c}_{x,l}(t) &= f_{\text{ROM,x},l}\big(c_{x,k}(t),c_{T,l}(t)\big), \label{eqn:PDEProjectionX}
\end{align}
where $f_{\text{ROM,x},l}$ is stated in Appendix A. 
The $n_x+n_T=n$ ODEs 
 \begin{align}\label{eqn:PODGalModel}
\dot{c}(t) &= \begin{bmatrix}
f_{\text{ROM,x},1}\big(c_{x,k}(t),c_{T,l}(t)\big)\\
\vdots\\
f_{\text{ROM,x},n_x}\big(c_{x,k}(t),c_{T,l}(t)\big)\\
f_{\text{ROM,T},1}\big(c_{x,k}(t),c_{T,l}(t)\big)\\
\vdots\\
f_{\text{ROM,T},n_T}\big(c_{x,k}(t),c_{T,l}(t)\big)\\
\end{bmatrix} =f_\text{ROM}\big(c(t) \big)
\end{align}
constitute the desired reduced order model for the drying process \eqref{eqn:PDE}, where
\begin{align}\label{eqn:PODGalStates}
c(t)&=[c_{x,1}(t) \hdots  c_{x,n_x}(t) \, c_{T,1}(t)  \hdots  c_{T,n_T}(t)]^\T ,
\end{align}
$c(t)\in\R^{n}$, refers to the state vector of the reduced order model.

The initial condition for \eqref{eqn:PODGalModel}, i.e., $c_{T,i}(t_0)$ and $c_{x,j}(t_0)$ for $i=1,\ldots,n_T$ and $j=1,\ldots,n_x$, can be taken from decomposition \eqref{eqn:SumLinApprox} and \eqref{eqn:SumLinApproxWater}, respectively, for $t_0$. 
The output \eqref{eqn:MeasuredOutput} is determined as follows.
 Solving \eqref{eqn:PODGalModel} for given initial conditions yields time series $c_{T,i}(t)$ and $c_{x,i}(t)$ that are used to approximate the temperature $T(y_i,t)$ and moisture $x(y_i,t)$ according to \eqref{eqn:SumLinApprox} and \eqref{eqn:SumLinApproxWater}, respectively.
 Collecting \eqref{eqn:SumLinApprox} and \eqref{eqn:SumLinApproxWater} as in \eqref{eqn:StateVectorCFD} yields 
 \begin{align}\label{eqn:StateApprox}
z(t)\approx \Phi c(t) + \bar{z},
\end{align}
which can be substituted in \eqref{eqn:MeasuredOutput} to obtain the desired output 
\begin{align}\label{eqn:OutputApproxGeneral}
w(t) &\approx \tilde{w}(t) = h\big(\Phi c(t) + \bar{z}\big).
\end{align}
In \eqref{eqn:StateApprox} and \eqref{eqn:OutputApproxGeneral}, the modes and time averages for temperature and moisture are collected in
\begin{align}\label{eqn:AllModes}
\Phi &= \begin{bmatrix}\begin{smallmatrix} 
\varphi_{x,1}(y_1)&\ldots&\varphi_{x,n_x}(y_1)&0&\ldots&0\\
\ldots&\ldots&\ldots&\ldots&\ldots&\ldots\\
\varphi_{x,1}(y_N)&\ldots&\varphi_{x,n_x}(y_N)&0&\ldots&0\\
0&\ldots&0&\varphi_{T,1}(y_1)&\ldots&\varphi_{T,n_T}(y_1)\\
\ldots&\ldots&\ldots&\ldots&\ldots&\ldots\\
0&\ldots&0&\varphi_{T,1}(y_N)&\ldots&\varphi_{T,n_T}(y_N)\\
\end{smallmatrix}\end{bmatrix} \nonumber\\&=\big[ \phi_1 \;\hdots \; \phi_n \big] \in\R^{M\times n},\\
\bar{z}&=[\bar{x}(y_1) \hdots  \bar{x}(y_N)  \, \bar{T}(y_1)  \hdots  \bar{T}(y_N)]^\T \in\R^M,\nonumber
\end{align}
respectively, where $M$ refers to the order of~\eqref{eqn:NonlinSysAllg}. 

\subsection{Observability Gramians based on reduced order models}\label{subsec:GramROM}
The reduced order model \eqref{eqn:PODGalModel} essentially results in two reductions. First, it is used to approximate eigenvalue problem \eqref{eqn:EigProblemLarge} for the Gramian of the original finite volume model by a much smaller eigenvalue problem. Secondly, the computational effort can be reduced considerably, because the perturbation direction $D_l e_i$ are chosen from the lower-dimensional space that results after the reduction. In fact, the dimension $n$ of the reduced order model is small enough to permit using all modes as perturbation directions.  
As a preparation, 
we convert the initial conditions for the finite volume model to initial conditions for the reduced order model with Corollary~\ref{lemma:InitialCondition}. 


\begin{corollary}\label{lemma:InitialCondition} 
  Consider the original finite volume model~\eqref{eqn:NonlinSysAllg} for the initial condition \eqref{eqn:GramInit} . 
	Let $D^*_l\in\R^{n\times M}$ and  $c_\mathrm{ss} \in\R^n$ be defined by
	\begin{align}
	D^*_{l,ij}&=\langle D_l e_j, \; \phi_i \rangle,  \label{eqn:GramInitPerturb}
	\\
	c_{\mathrm{ss},i} &= \langle  z_\mathrm{ss} - \bar{z}, \; \phi_i \rangle. \label{eqn:GramInitSs}
	\end{align}
	Then
	\begin{align}\label{eqn:GramInitROMHigh}
	c_{0}=h_d D^*_l e_j +  c_\mathrm{ss}  \in\R^n,
	\end{align}
	is the  initial condition for the corresponding reduced order model~\eqref{eqn:PODGalModel}. 
\end{corollary}
\begin{proof}
	As in \eqref{eqn:TimeCoeff}, the reduced model states are obtained by projecting the original states onto the modes. The $i$th entry of $c_0$ reads
	\begin{align*}
	c_{0,i}= \langle z_0- \bar{z}, \; \phi_i \rangle.
	\end{align*}
	Substituting \eqref{eqn:GramInit} yields
	\begin{align}\label{eqn:GramInitROMgeneral}
	c_{0,i}&= \langle h_d D_l e_j +  z_\mathrm{ss} - \bar{z}, \; \phi_i \rangle \nonumber\\
	&= h_d \underbrace{\langle D_l e_j, \; \phi_i \rangle}_{=D_{l,ij}^*} + \underbrace{\langle  z_\mathrm{ss} - \bar{z}, \; \phi_i \rangle}_{=c_{\mathrm{ss},i}}
	\end{align}
	due to the linearity of the inner product.
\end{proof}
Solving the reduced model~\eqref{eqn:PODGalModel} with the initial condition from Corollary~\ref{lemma:InitialCondition} results in the desired approximation $\tilde{w}_{dli}(t)$ of the original output response $w_{dli}$
	\begin{align}\label{eqn:OutputApprox}
	w_{dli}(t) &\approx \tilde{w}_{dli}(t) = h\big(\Phi c_{dli}(t) + \bar{z}\big),
	\end{align}
  where $c_{dli}(t)$ is the solution of the reduced model \eqref{eqn:PODGalModel} for initial condition $c_0$. 
However, if $s\cdot r\cdot M$ initial conditions $z_0$ are constructed for the original model as in Section~\ref{subsec:CCM}, $s\cdot r\cdot M$ initial conditions $c_0$ for the reduced model result with Corollary~\ref{lemma:InitialCondition}. 
Since the initial conditions for the reduced model can only be chosen from an $n$-dimensional space with $n\ll M$, 
it is not efficient to simply transform all initial conditions with~\eqref{eqn:GramInitROMHigh}. We show in Proposition~\ref{Prp:GramROM} that choosing the modes as initial conditions for the original model results in particularly simple $n$ independent initial conditions for the reduced model. 
Corollary~\ref{lem:ExtensionWithFundamentalTheorem} is required as a preparation.
\begin{corollary}\label{lem:ExtensionWithFundamentalTheorem}
	Let the columns of $H=[\eta_1 \hdots\, \eta_{M-n}]$, $H\in\R^{M\times (M-n)}$ be an orthonormal basis for $\mathrm{ker}(\Phi^\T)$ and $l\in\{-1, 1\}$.
	Then 
	\begin{align}\label{eqn:PerturbDirection}
	D_l = \big[ (-1)^l \sqrt{\Delta V}\Phi,\; H\big]
	\end{align}
	is orthonormal with dimension $M\times M$. 
\end{corollary}
\begin{proof}
Since $\Phi\in\R^{M\times n}$ and $H\in\R^{M\times (M-n)}$ by construction, we have $D_l\in\R^{M\times M}$. Furthermore, 
\begin{align*}
D_l^\T D_l = \begin{bmatrix} \Delta V \Phi^\T \Phi & (-1)^l\sqrt{\Delta V}\Phi^\T  H \\ (-1)^l\sqrt{\Delta V}H^\T \Phi & H^\T H\end{bmatrix}= I_M
\end{align*}
since $\Delta V \Phi^\T \Phi=I_n$ due to \eqref{eqn:OrthoNormality} and $H^\T H=I_{M-n}$ by construction and,
due to the fundamental theorem of linear algebra, 
$\mathrm{im}(\Phi)^\orthcomp=\mathrm{ker}(\Phi^\T)=\mathrm{im}(H)$, thus $H^\T\Phi=0_{(M-n)\times n}$ and $\Phi^\T H = 0_{n \times (M-n)}$. 
\end{proof}

We use $+\Phi$ and $-\Phi$, i.e., $(-1)^l \Phi$, $l=1,2$, in \eqref{eqn:PerturbDirection} to better account for nonlinear system behavior as proposed in \cite{Lall1999}. 
Having augmented the modes $\Phi\in\R^{M\times n}$ to the orthonormal matrix $D_l\in\R^{M\times M}$ in \eqref{eqn:PerturbDirection}, 
$D_l$ can now serve as a perturbation matrix in \eqref{eqn:Gram}. 
\begin{proposition}\label{Prp:GramROM} Let 
$G_o$ be the observability Gramian \eqref{eqn:Gram} that results for arbitrary nonzero $h_d\in\R$ and $D_l$ from \eqref{eqn:PerturbDirection}.
Let
\begin{align}
\tilde{h}_d&=\sqrt{\Delta V}h_d, \, 
\tilde{D}_l = (-1)^l I_n, \hspace{.25cm} l=1,2\label{eqn:GramROMh}
\end{align}
where $I_n$ is the identity matrix of size $\R^{n\times n}$. 
Furthermore, let 
\begin{align}\label{eqn:GramROM}
W_o=\sum_{l=1}^r \sum_{d=1}^s\frac{1}{rs\tilde{h}_d^2}\int_0^\infty   \tilde{D}_l \tilde{\Psi}_{dl} \tilde{D}_l^\T \diff t,
\end{align}
$W_o \in \R ^{n \times n}$, refer to the observability Gramian for the reduced order model with
\begin{align}\label{eqn:GramROMCov}
\tilde{\Psi}_{dl,ij}=\big(\tilde{w}_{dli}(t)-\tilde{w}_{\mathrm{ss},dli}\big)^\T\big(\tilde{w}_{dlj}(t)-\tilde{w}_{\mathrm{ss},dlj}\big),
\end{align}
where $\tilde{w}_{dli}$ refers to output \eqref{eqn:OutputApprox} for initial condition
\begin{align}\label{eqn:GramROMInit}
c_0=\tilde{h}_d \tilde{D}_l \tilde{e}_i +  c_\mathrm{ss}
\end{align}
$\tilde{w}_{\mathrm{ss},dli}= \lim\limits_{t \to \infty} \tilde{w}_{dli}(t)$ and $c_\mathrm{ss}$ in \eqref{eqn:GramROMInit} is defined in \eqref{eqn:GramInitSs}.

Then $G_o$ can be approximated by
\begin{align}
\tilde{G}_o= \Delta V ^2 \Phi W_o \Phi^\T.\label{eqn:GramApproxSmall}
\end{align}
\end{proposition}

\begin{proof}
Substituting \eqref{eqn:PerturbDirection} into \eqref{eqn:GramInitPerturb} yields 
\begin{align}\label{eqn:GramPertubCases}
D^*_{l,ij}&=\langle \big[ (-1)^l \sqrt{\Delta V}\Phi,\; H\big] e_j, \; \phi_i \rangle \nonumber\\
&=\begin{cases} (-1)^l \sqrt{\Delta V} \langle \phi_j , \;  \phi_i \rangle   &\text{for } j=1,\ldots, n \\ 0&\text{for } j=n+1,\ldots, M\end{cases},\nonumber \\
&=\begin{cases} (-1)^l  \sqrt{\Delta V} \delta_{i,j}  &\text{for } j=1,\ldots, n \\ 0&\text{for } j=n+1,\ldots, M\end{cases},
\end{align}
for any $i=1,\ldots,n$ and $l=1,2$. The element-wise definition \eqref{eqn:GramPertubCases} yields 
\begin{align}\label{eqn:GramPertubMatrixLarge}
D_l^*=  \sqrt{\Delta V} \big[ (-1)^l I_n, \, 0_{M-n} \big] \in\R^{n\times M},
\end{align}
where $I_n$ is the unity matrix of size $\R^{n\times n}$ and $0_{M-n}$ is a zero matrix of size $\R^{n \times (M-n)}$. 
Solving the reduced order model \eqref{eqn:PODGalModel} for initial state \eqref{eqn:GramInitROMHigh} and $D^*_l $ from \eqref{eqn:GramPertubMatrixLarge} yields $M$ time series $\tilde{w}^*_{dli}(t)$, $i=1,\ldots,M$ and $\tilde{w}^*_{\mathrm{ss},dli}= \lim\limits_{t \to \infty} \tilde{w}^*_{dli}(t)$ that are used to determine 
\begin{align}\label{eqn:GramRomCovLarge}
\tilde{\Psi}^*_{dl,ij}=\big(\tilde{w}^*_{dli}(t)-\tilde{w}^*_{\mathrm{ss},dli}\big)^\T\big(\tilde{w}^*_{dlj}(t)-\tilde{w}^*_{\mathrm{ss},dlj}\big)
\end{align}
$i,j=1,\ldots,M$, according to \eqref{eqn:GramCov}. Since $\tilde{w}^*_{dlj}(t) \approx w_{dlj}(t)$ according to  \eqref{eqn:OutputApprox}, we have $\tilde{\Psi}^*_{dl}\approx \Psi_{dl}$ and thus 
\begin{align}\label{eqn:GramApproxLarge}
G_o\approx\tilde{G}_0=\sum_{l=1}^r \sum_{d=1}^s\frac{1}{rsh_d^2}\int_0^\infty   D_l \tilde{\Psi}_{dl}^* D_l^\T \diff t,
\end{align}
with $D_l$ according to \eqref{eqn:PerturbDirection}.

We now show the first $i,j=1,\ldots, n$ entries of \eqref{eqn:GramRomCovLarge} suffice to determine \eqref{eqn:GramApproxLarge}. 
For $i=n+1,\ldots, M$, we have $D^*_l e_i=0$ according to \eqref{eqn:GramPertubMatrixLarge}. 
We obtain $c_0=c_\mathrm{ss}$ for the initial state \eqref{eqn:GramInitROMHigh} in this case. In other words, the initial state is a steady state of the system. Solving the reduced order model \eqref{eqn:PODGalModel} for $c_0=c_\mathrm{ss}$ results in $c(t)=c_\mathrm{ss}=\mathrm{const}$ for all $t\geq 0$. Thus, the system output yields $\tilde{w}^*_{dli}(t) = \tilde{w}^*_{\mathrm{ss},dli}=\mathrm{const}$ for all $t\geq 0$. Consequently, $\tilde{\Psi}^*_{dl,ij}=0$ in \eqref{eqn:GramRomCovLarge} for $i,j=n+1,\ldots, M$. We can therefore neglect all time series $\tilde{w}^*_{dli}(t) $ for $i>n$ since they do not contribute to $\tilde{G}_o$. In other words, $\tilde{G}_o$ can be determined from only the first $i,j=1,\ldots, n$ elements of $\tilde{\Psi}^*_{dl}$. 
Let $\tilde{\Psi}_{dl} \in\R^{n\times n}$ consist of these elements, i.e., $\tilde{\Psi}_{dl,ij}=\tilde{\Psi}^*_{dl,ij}$ for $i,j=1,\ldots, n$. Substituting \eqref{eqn:PerturbDirection} in \eqref{eqn:GramApproxLarge} results in
\begin{align*}
\tilde{G}_0&=\sum_{l=1}^r \sum_{d=1}^s\frac{1}{rsh_d^2}\int_0^\infty  \Big( (-1)^l \sqrt{\Delta V}\Phi \Big) \tilde{\Psi}_{dl} \Big( (-1)^l \sqrt{\Delta V}\Phi \Big)^\T \diff t,\\
&=\Delta V\Phi \sum_{l=1}^r \sum_{d=1}^s\frac{1}{rsh_d^2}\int_0^\infty  \Big((-1)^l I_n \Big) \tilde{\Psi}_{dl} \Big((-1)^l{I_n}\Big)^\T \diff t\; \Phi^\T.
\end{align*}
Substituting $\tilde{D}_l$ and $h_d$ from \eqref{eqn:GramROMh} yields
\begin{align}
\tilde{G}_0&=\Delta V^2 \Phi \sum_{l=1}^r \sum_{d=1}^s\frac{1}{rs \tilde{h}_d^2}\int_0^\infty  \tilde{D}_l \tilde{\Psi}_{dl} {\tilde{D}_l}^\T \diff t\; \Phi^\T\nonumber\\
&=\Delta V^2\Phi W_o \Phi^\T, 
\end{align}
which is the claim. 
\end{proof}

We showed that \eqref{eqn:GramApproxSmall} approximates \eqref{eqn:Gram}. Then 
\begin{align}\label{eqn:EigProblemLargeApprox}
\tilde{G}_o \tilde{v}_k - \tilde{\beta}_k \tilde{v}_k = 0
\end{align}
approximates eigenvalue problem \eqref{eqn:EigProblemLarge}, i.e., 
\begin{align}\label{eqn:LinApproxEig}
\beta_k \approx \tilde{\beta}_k \;\text{ and }\;
v_k \approx \tilde{v}_k. 
\end{align}
The eigenvalue problem for $G_o\in\R^{M\times M}$ can be replaced by the much smaller one for $W_o\in\R^{n\times n}$. 
This is stated concisely in the following corollary. 
\begin{corollary}\label{Cor:SmallEigProblem}The non-zero eigenvalues $\tilde{\beta}_k$ in \eqref{eqn:EigProblemLargeApprox} are equal to those of the $n$-dimensional eigenvalue problem
\begin{align}\label{eqn:EigProblemSmall}
  \Delta V W_o \nu_k -  \tilde{\beta}_k \nu_k = 0
\end{align}
and the respective eigenvectors $\tilde{v}_k $ in \eqref{eqn:EigProblemLargeApprox} are given by
\begin{align}\label{eqn:EigVectorLargeFromSmall}
\tilde{v}_k =\Phi \nu_k.
\end{align}
\end{corollary}

\begin{proof} Substituting \eqref{eqn:GramApproxSmall} into \eqref{eqn:EigProblemLargeApprox} yields 
\begin{align}\label{eqn:CharPolyEigProb}
&\det\big( \tilde{\beta}_k I_M - \Delta V^2 \Phi W_o\Phi^T \big) \nonumber\\&= \tilde{\beta}_k^{M-n} \det\big( \tilde{\beta}_k I_n - \Delta V^2 W_o \Phi^T\Phi\big),
\end{align}
with Sylvester's determinant identity, where $I_M$ and $I_n$ are the $\R^{M\times M}$ and $\R^{n\times n}$ identity matrices, respectively. The right hand side of \eqref{eqn:CharPolyEigProb} can further be simplified to 
\begin{align}\label{eqn:CharPolyEigProbSimple}
\tilde{\beta}_k^{M-n} \det\big( \tilde{\beta}_k I_n - \Delta V W_o\big),
\end{align}
with $\Phi^T \Phi= \text{diag}(\nicefrac{1}{\Delta V}, \dots, \nicefrac{1}{\Delta V})$, which follows from~\eqref{eqn:OrthoNormality}. 
The non-trivial roots $\tilde{\beta}_k$ of \eqref{eqn:CharPolyEigProbSimple}, i.e., the eigenvalues of~\eqref{eqn:EigProblemSmall}, correspond to the non-zero roots of the left hand side of \eqref{eqn:CharPolyEigProb}, i.e., the non-zero eigenvalues of \eqref{eqn:EigProblemLargeApprox}, which establishes the claim for the eigenvalues. 
Substituting $\Phi^T \Phi= \text{diag}(\nicefrac{1}{\Delta V}, \dots, \nicefrac{1}{\Delta V})$ into~\eqref{eqn:EigProblemSmall} and left-multiplying with $\Phi$ yields
\begin{align}\label{eqn:EigProblemSmallLarge2}
  \Delta V^2 \Phi W_o  \Phi^\T  \Phi \nu_k  -  \tilde{\beta}_k \Phi \nu_k = 0.
\end{align}
The eigenvalue problem \eqref{eqn:EigProblemLargeApprox} results from substituting $\Phi \nu_k$ in \eqref{eqn:EigProblemSmallLarge2} by \eqref{eqn:EigVectorLargeFromSmall}, which proves the claim on the eigenvectors. 
\end{proof}

\section{Application to the drying process of wood chips}\label{sec:ApplicationToWoodChip}

We apply the POD-Galerkin based model reduction procedure presented in Section~\ref{subsec:modelreduction} to the wood chip drying process introduced in Section~\ref{sec:WoodChipModelling}. We briefly evaluate the reduced model in Section~\ref{subsec:ROMevaluation} as necessary for the observability analysis in Section~\ref{subsec:OBSVDryingProcess}. 
We also analyze the influence of the degree of reduction on the observability analysis.

\subsection{Reduced order model fitness}\label{subsec:ROMevaluation}


The reduced order model is based on finite volume simulation results for a drying process under typical conditions. An initially wet wood particle with $x_0=0.8$ at room temperature $T_0=\unit[298.15]{K}$ is exposed to hot dry air until a steady state is reached after about $\unit[1100]{s}$. The finite volume model \eqref{eqn:NonlinSysAllg} yields the temperature $T(y_i,t_j)$ and moisture $x(y_i,t_j)$ on a temporal and spatial grid $i=1,\ldots,N$, $j=0,\ldots,m-1$ for $N=1000$ grid points and $m=100$ time steps.
Applying POD to these simulation results yields modes $\varphi_{x,l}(y_i)$, $\varphi_{T,k}(y_i)$ and coefficients $c_{x,l}(t)$, $ c_{T,k}(t)$, $l=1,\ldots,n_x$, $k=1,\ldots,n_T$ such that \eqref{eqn:SumLinApprox} and \eqref{eqn:SumLinApproxWater} hold. We choose cut-off values of $n_x=n_T=5$, thus $n=10$, to ensure \eqref{eqn:Energy} satisfies $E(n_x), \, E(n_T)>0.9999$. Figure \ref{fig:PODCoeff} shows the first $k=1,\ldots,5$ coefficients $c_{x,l}(t_j)$ and $c_{T,k}(t_j)$ for moisture and temperature, respectively. 

\begin{figure}[h!]
	\centering
	\subfigure{\includegraphics[width=0.35\textwidth]{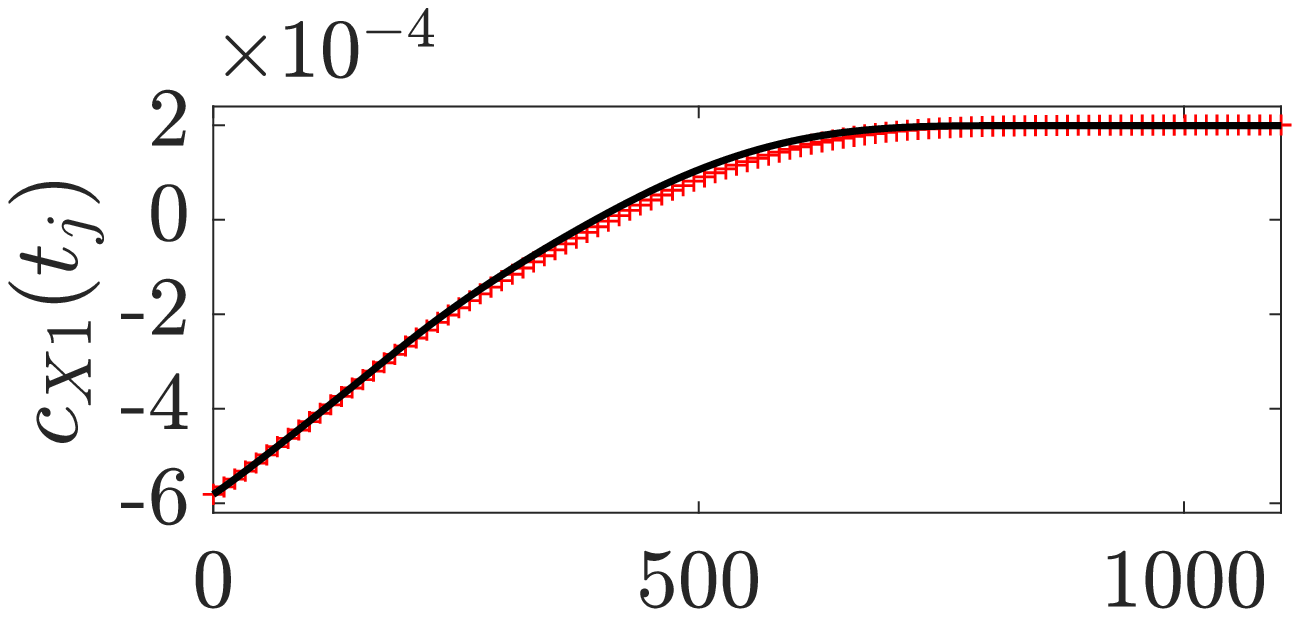}}
	\subfigure{\includegraphics[width=0.35\textwidth]{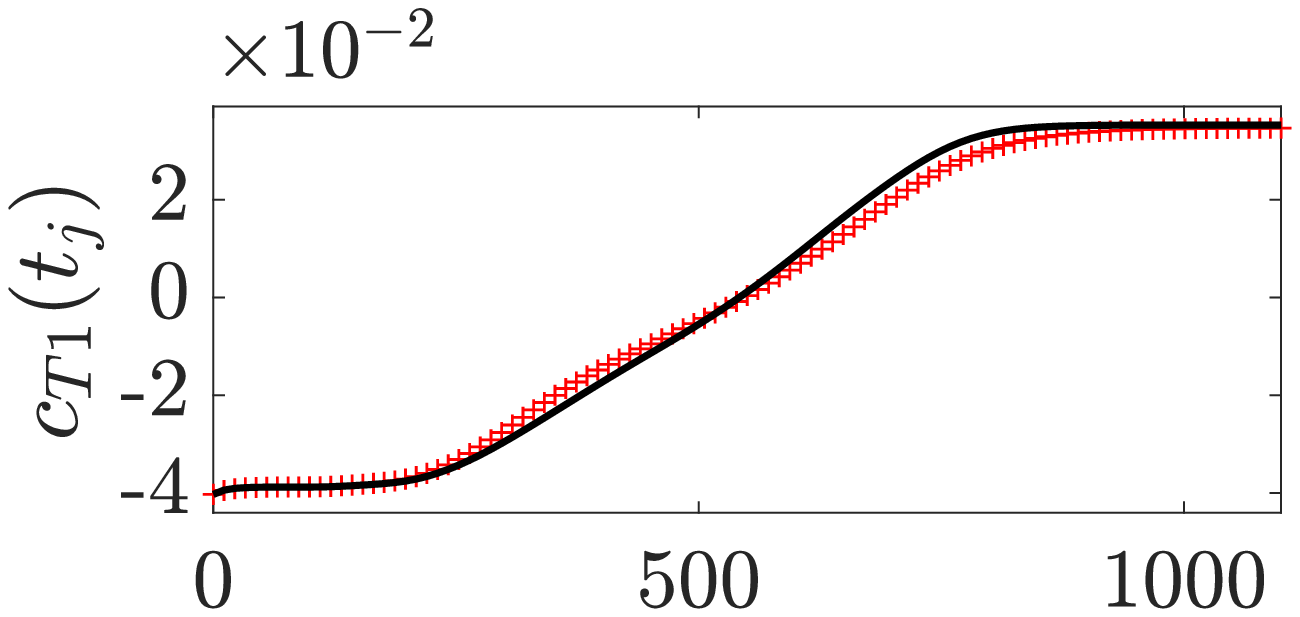}}
	\subfigure{\includegraphics[width=0.35\textwidth]{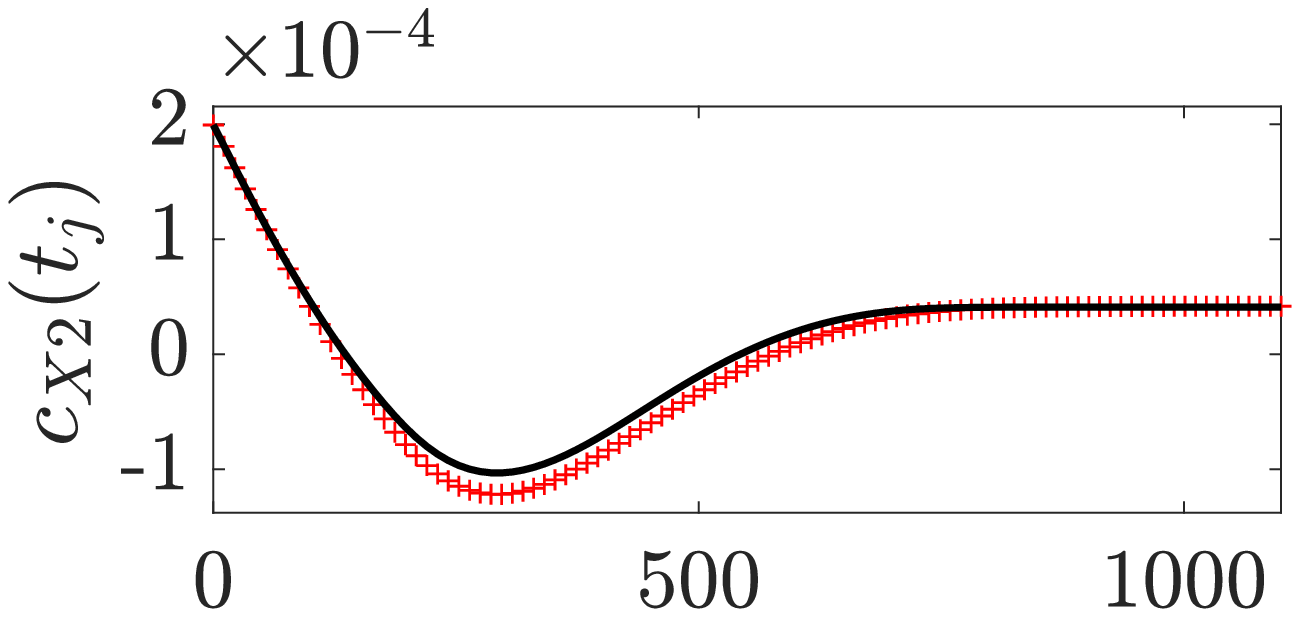}}
	\subfigure{\includegraphics[width=0.35\textwidth]{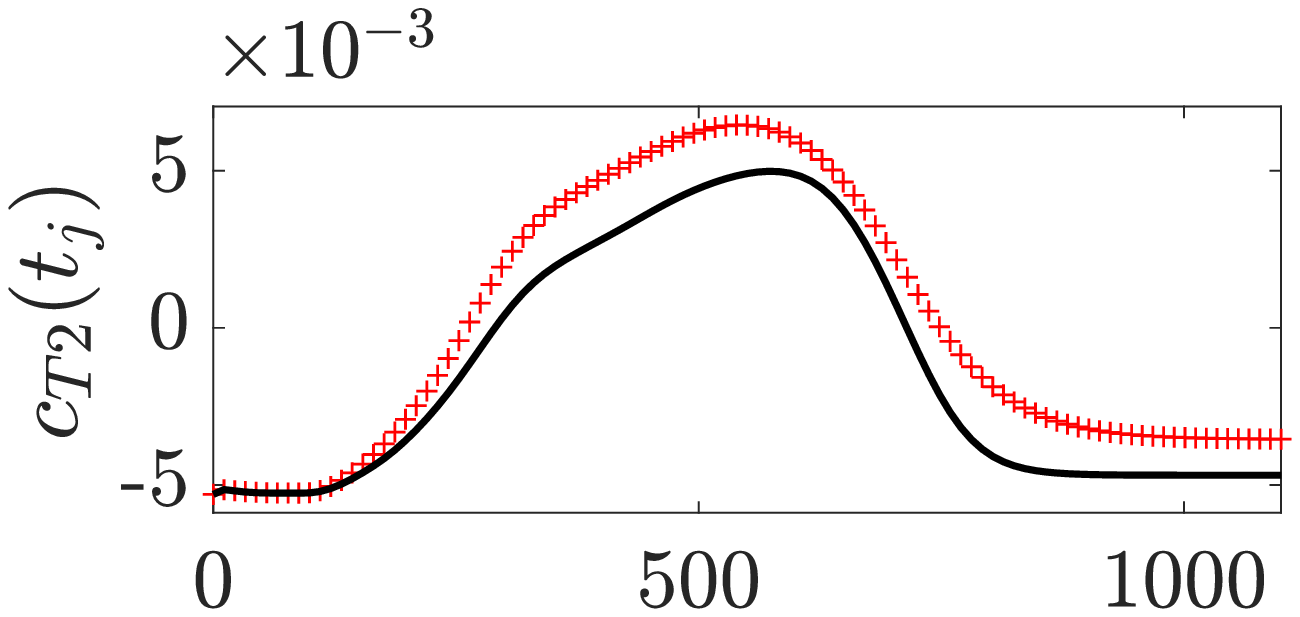}}
	\subfigure{\includegraphics[width=0.35\textwidth]{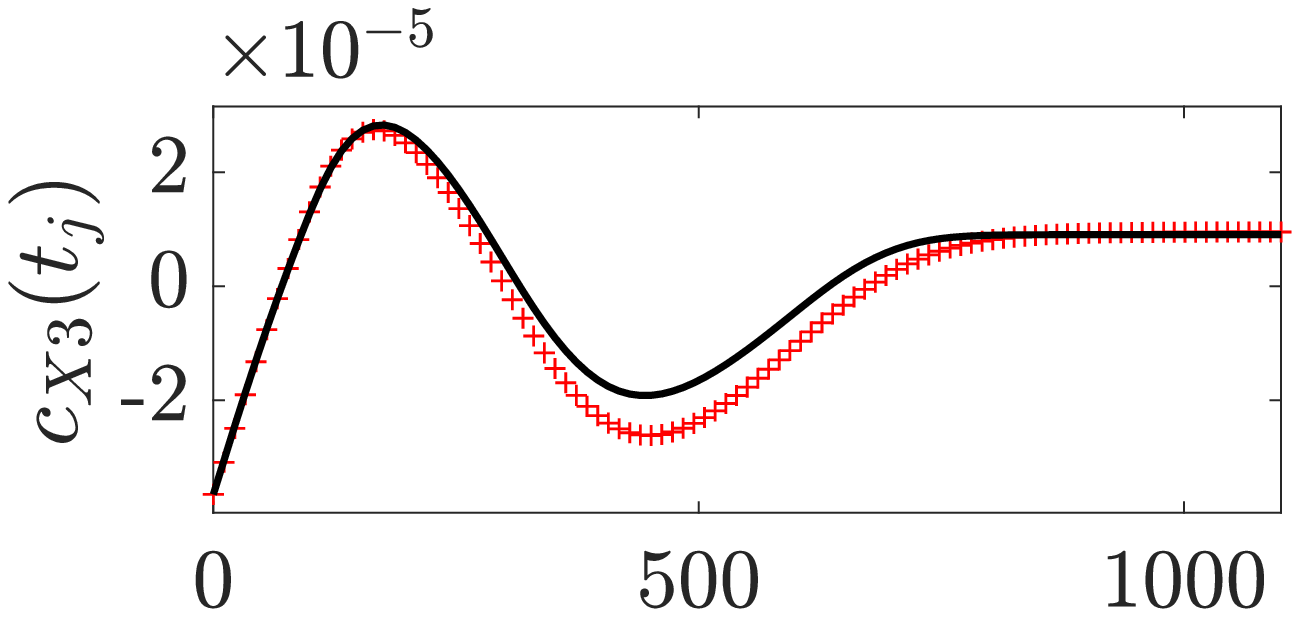}}
	\subfigure{\includegraphics[width=0.35\textwidth]{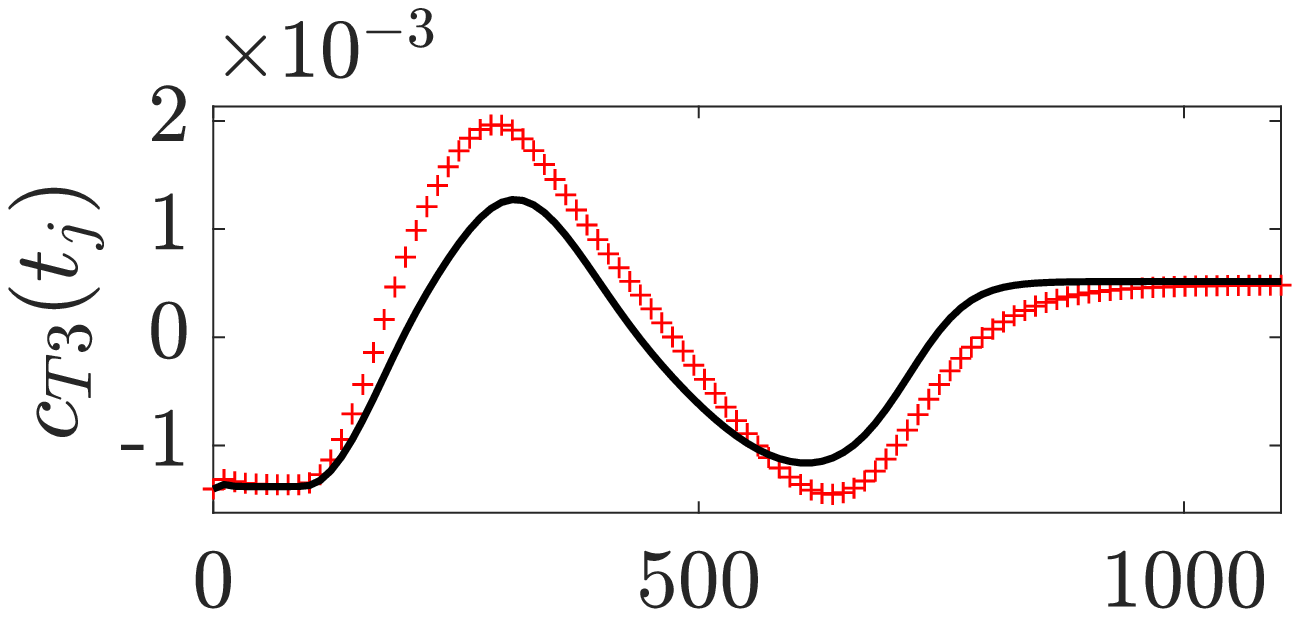}}
	\subfigure{\includegraphics[width=0.35\textwidth]{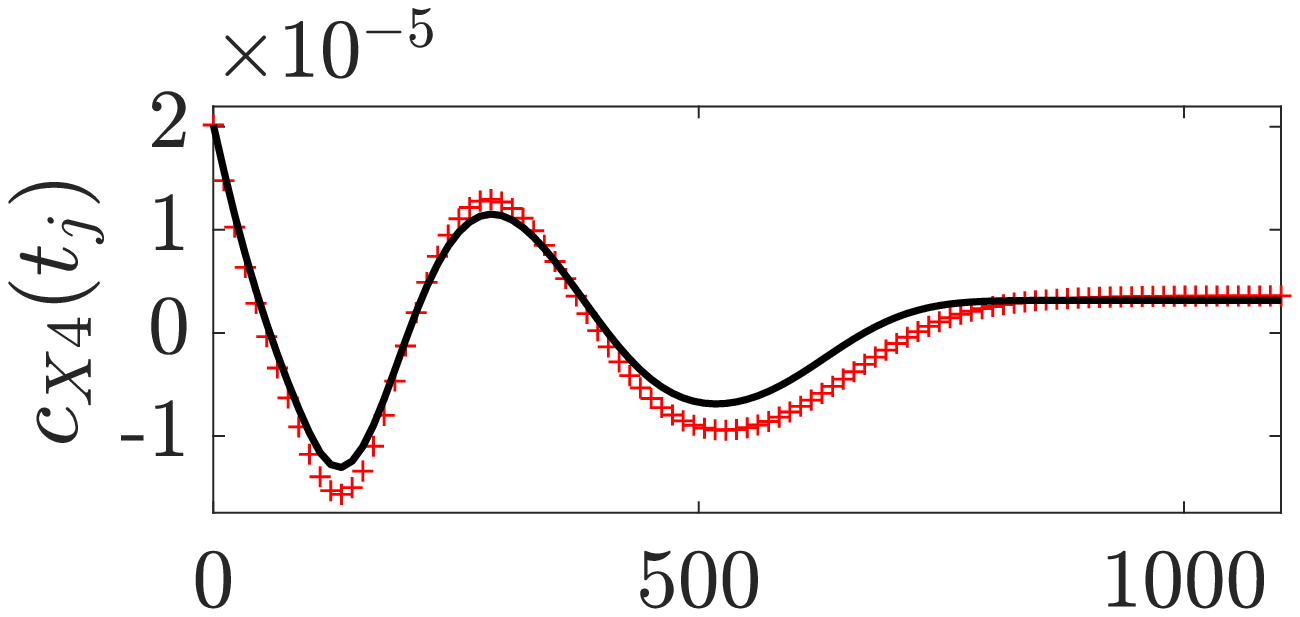}}
	\subfigure{\includegraphics[width=0.35\textwidth]{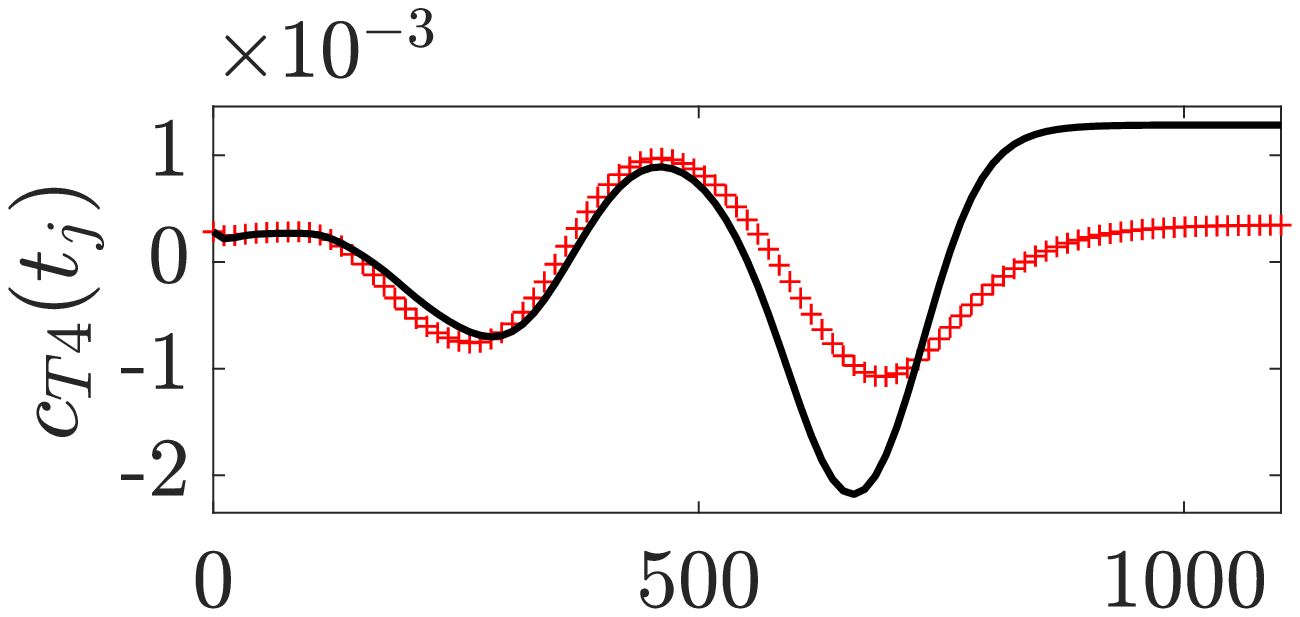}}
	\subfigure{\includegraphics[width=0.35\textwidth]{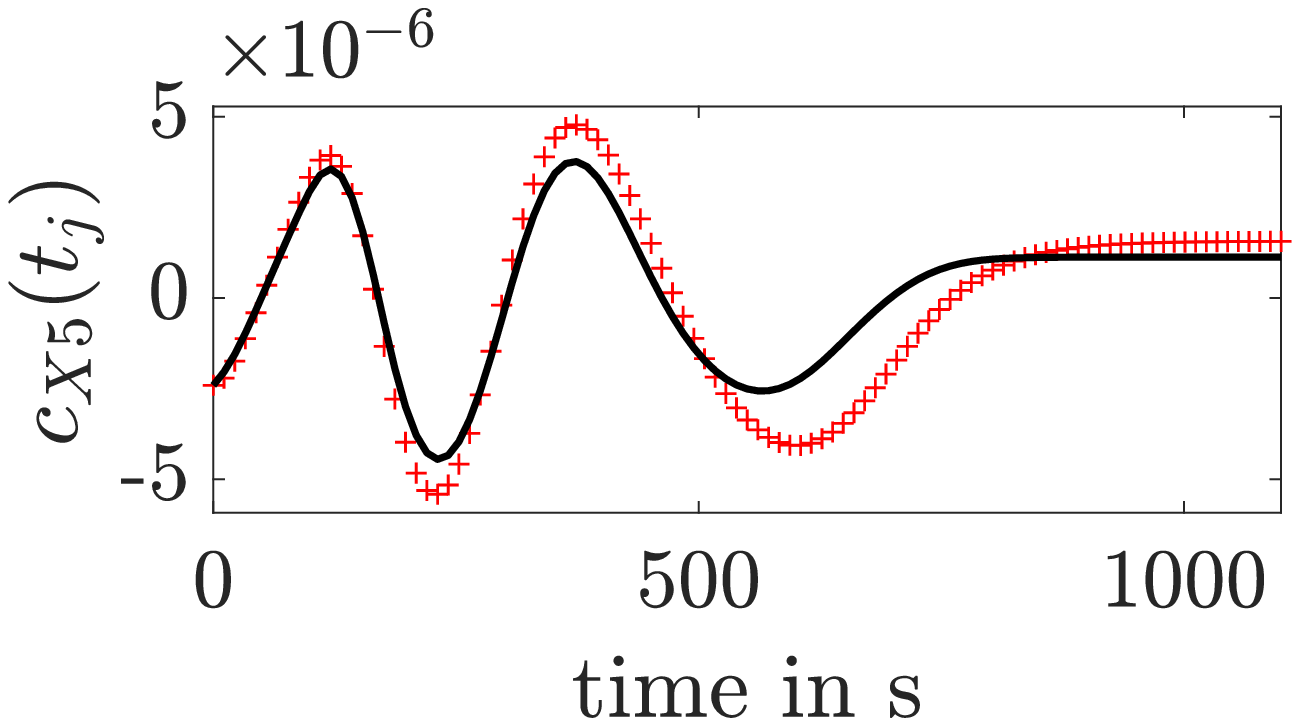}}
	\subfigure{\includegraphics[width=0.35\textwidth]{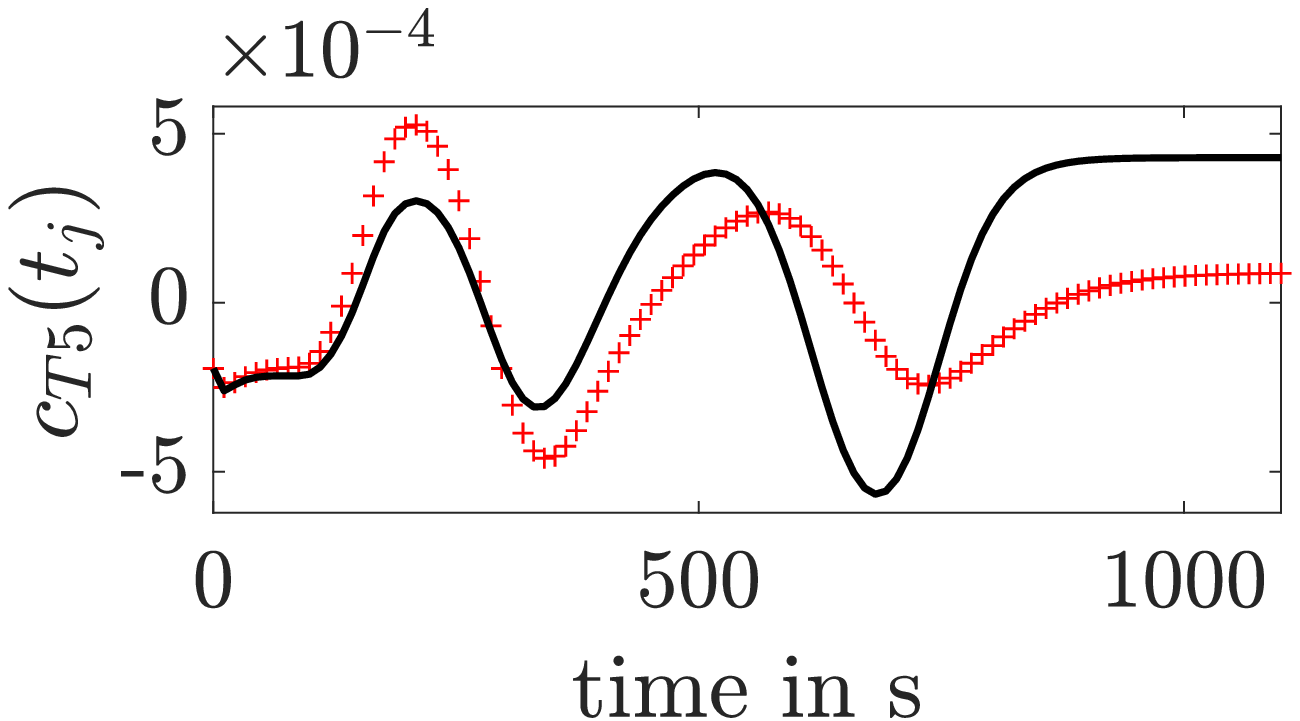}}
		\caption{Coefficients from approximation \eqref{eqn:SumLinApproxWater} (left) and \eqref{eqn:SumLinApprox} (right). Time-continuous results of a reduced order model with $n=10$ (black) are compared to those of the original finite volume simulation (red).}
	\label{fig:PODCoeff}
\end{figure}

We first check the approximation quality of the reduced order model. Figure~\ref{fig:PODCoeff} compares the time-dependent coefficients obtained by the POD to the 
time series for $c_{x,i}(t)$ and $c_{T,j}(t)$ defined in~\eqref{eqn:PODGalStates} that result from solving the reduced order model \eqref{eqn:PODGalModel}. 
We observe that the most important modes match acceptably well while some deviations must be accepted for higher-order modes.

Substituting $c_{x,i}(t)$ and $c_{T,j}(t)$ from~\eqref{eqn:PODGalStates} into \eqref{eqn:SumLinApprox} and \eqref{eqn:SumLinApproxWater}
yields the desired  moisture distribution $x(y, t)$ and temperature distribution $T(y, t)$. We determine the normalized root mean square error
\begin{align}\label{eqn:MeanRelativeError}
\varepsilon_T &= \frac{\sqrt{\frac{1}{mN} \sum_{i=1}^{N}\sum_{j=1}^{m} \big(  \epsilon_T(y_i,t_j) \big)^2 } }{\max\limits_{i,j}(T(y_i,t_j))-\min\limits_{i,j}(T(y_i,t_j))},
\end{align}
over all locations $y_i$, $i= 1, \dots, N$ and times $t_j$, $j= 1, \dots, m$, where 
\begin{align}\label{eqn:AbsError}
\epsilon_T(y_i,t_j) &= T(y_i,t_j) - \big( \bar{T}(y_i) + \textstyle\sum_{k=1}^{n_T} \varphi_{T,k}(y_i) c_{T,k}(t_j) \big).
\end{align}
The error amounts to $\varepsilon_T=3.6\%$.
The normalized root mean square error for the moisture $\varepsilon_x$ is determined accordingly and amounts to $\varepsilon_x=1.9\%$. 

\begin{figure}[t]
	\centering
	\subfigure{\includegraphics[width=0.7\textwidth]{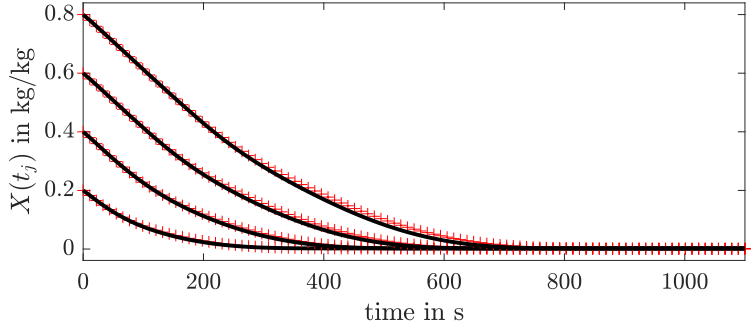}}
	\caption{The total moisture $X(t)$ determined with the reduced order model (solid black) is compared to the total moisture that was obtained with finite volume methods (red) for different initial conditions $x_0\in\{0.8,\;0.6,\;0.4,\;0.2\}$.}
	\label{fig:EKFEstimatedOutput}
\end{figure}

The overall moisture \eqref{eqn:TotalMoisture}, specifically its discrete variant 
\begin{align}\label{eqn:TotalMoistureDiscrete}
X(t) \approx \frac{1}{V}\sum_{i=1}^N x(y_i,t_j) \Delta V = \tilde{X}(t_j),
\end{align}
obtained with the reduced order model is compared to the total moisture obtained with the original model for different initial moistures $x_0\in\{0.8,\;0.6,\;0.4,\;0.2\}$ in Figure~\ref{fig:EKFEstimatedOutput}. We analyze different initial moistures since a variation of initial conditions is required in the observability analysis. Note that the reduced order model was derived from finite volume simulation results for $x_0=0.8$ only; the finite volume simulation results for $x_0\in\{0.6,\;0.4,\;0.2\}$ are shown for validation purpose only.
 We observe a good agreement again. 
The normalized root mean square error for the total moisture
\begin{align*}
\varepsilon_X &= \frac{\sqrt{\frac{1}{m} \sum_{j=1}^{m} \big(  X(t_j) -\tilde{X}(t_j) \big)^2 } }{\max\limits_{j}(X(t_j))-\min\limits_{j}(X(t_j))},
\end{align*}
with $\tilde{X}(t_j)$ according to \eqref{eqn:TotalMoistureDiscrete}, is not greater than $\varepsilon_X=1.1\%$ for all initial conditions. 
We conclude the errors are acceptable and the reduced order model sufficiently accurately represents the inner particle temperature and moisture distribution during the drying process.

We stress again the simplifications for the explanations in Section~\ref{subsec:modelreduction} do not apply in the present Section~\ref{sec:ApplicationToWoodChip}. In particular, the material parameters $s$, $\lambda$ and $\delta$ are nonlinear functions of the local moisture or temperature approximations and $\lambda$ and $\delta$ are tensors ($\R^{3\times 3}$) that model the wood anisotropy.

\subsection{Choice of measurement position}\label{subsec:OptimalMeasurement}

It is ultimately our aim to determine the inner particle moisture distribution from measurements of only the surface temperature, where these measurements may even be available at a few points on the surface only. 
We investigate the limiting case of the observability from a temperature measurement at a single, isolated point on the surface in the present subsection. 
The subsequent subsection addresses the observability from measurements on a partial face of a particle, which models a particle that is partly covered by its neighbors.


 \begin{figure}[b]
	\centering
	\subfigure{\includegraphics[width=0.35\textwidth]{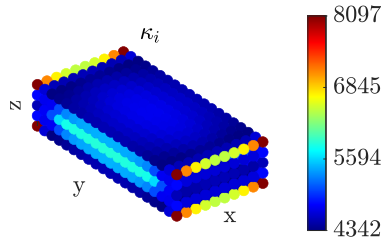}}
	\caption{Observability measure $\kappa_i$ as a function of the measurement position.}
	\label{fig:ObservabilityMeasure}
\end{figure}

Let $\kappa_i$, $i=1,\ldots,N_w$ refer to the observability measure~\eqref{eqn:ObsvMeausure} for the output $w(t)=T(y_i,t)$, where $y_i$ now is a single arbitrary point on the surface of the particle. 
$N_w=568$ such points exist. The eigenvalues required to determine $\kappa_i$ are obtained from the reduced order model (ROM) based observability Gramian $W^*_{\text{o},i}$ according to Corollary~\ref{Cor:SmallEigProblem} for each of these points. 
Specifically, we apply initial conditions~\eqref{eqn:GramROMInit} to the reduced order model \eqref{eqn:PODGalModel} of order $n=10$, where $\tilde{h}_d\in\{ 10^{-7},10^{-6},10^{-5}\}$, i.e., $s= 3$, $\tilde{D}_l = (-1)^l I_n$ with $r=2$ and where $c_\text{ss}$ refers to the steady state that results for an ambient temperature of $T_\infty=\unit[298.15]{K}$. Note that these $h_d$ cover $3$ orders of magnitude. Solving the reduced order model for each of the $s\cdot r\cdot n=60$ initial states $c_0$ results in the desired output time series $\tilde{w}_{dli}(t_j)$ and the steady state output value $\tilde{w}_{\text{ss},dli}$.
Approximating the integral in~\eqref{eqn:GramROM} by a sum with $\Delta t=\unit[0.005]{s}$ and $m_\text{f}=1\times 10^6$ time steps yields
\begin{align}\label{eqn:GramROMDiscrete}
W_{o,i}^*=\sum_{l=1}^r \sum_{d=1}^s\frac{1}{rsh_d^2} \sum_{j=0}^{m_\text{f}}  I_n  \tilde{\Psi}_{dl} I_n^\T \Delta t,
\end{align}
with 
\begin{align*}
\tilde{\Psi}_{dl,ij}=\big(\tilde{w}_{dli}(t_j)-\tilde{w}_{\mathrm{ss},dli}\big)^\T\big(\tilde{w}_{dlj}(t_j)-\tilde{w}_{\mathrm{ss},dlj}\big)
\end{align*}
for each surface point $i= 1, \dots, N_w$. 
The value of $m_\text{f}$ was chosen sufficiently large for $w_{dli}(t_{m_\text{f}})$ to reach its steady state value. 
Note that $W_{o,i}^*\in\R^{10 \times 10}$, since the order of the reduced order model is $n=10$.

Figure~\ref{fig:ObservabilityMeasure} shows the resulting values for $\kappa_i$, $i= 1, \dots, N_w$.
It is evident that the best measurement positions occur at some of the edges of the particle. The surfaces orthogonal to the fiber direction of the wood also provide good measurement positions. 
From a practical point of view, the particle edges are not suitable for measurements due to their narrow extension. 
We conclude the surfaces pointing in fiber direction should be used for measurements.



\subsection{Observability of the drying process}\label{subsec:OBSVDryingProcess}

\begin{figure}[t!]
	\centering
	\subfigure{\includegraphics[width=0.35\textwidth]{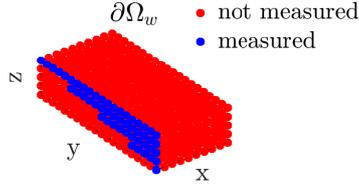}}
	\caption{The domain colored in blue marks the measured surface $\partial\Omega_w$.}
	\label{fig:MeasuredSurface}
\end{figure}

From here on, the temperatures are assumed to be measured on the surfaces orthogonal to the fiber direction. More precisely, we assume the average surface temperature 
\begin{align}\label{eqn:Output}
w(t)&= \tfrac{1}{v} \big(T(y_1,t)+  \hdots +  T(y_v,t)\big) \, \in\R,
\end{align}
to be measured as a function of time for the blue domain in Figure~\ref{fig:MeasuredSurface}, which is composed of $v=51$ surface elements. 
This domain models a particle face that is partly covered by neighboring particles. 
 
We first determine the observability Gramian according to \eqref{eqn:GramROMDiscrete} for output \eqref{eqn:Output} and $h_d$, $D_l$, $m_f$ and $\Delta t$ as in Section~\ref{subsec:OptimalMeasurement}. The observability measure \eqref{eqn:ObsvMeausure} for this case reads $\kappa=4.5 \times 10^3$. 
We conclude the reduced order model of order $n=10$ is suitable to be used in a state observer since the output \eqref{eqn:Output} permits observing the states of the reduced order model thus permits reconstructing the temperature and moisture distribution inside the wood chip. 

We additionally determine the first eigenvector $\tilde{v}_1$ indicating the most observable direction (Figure~\ref{fig:MostControllableModes}). We observe that the temperature is significantly better observable than the moisture. We therefore expect the Kalman filter presented in Section~\ref{sec:EKF} to yield better results for the temperature than the moisture reconstruction.

\begin{figure}[t]
	\centering
	\subfigure{\includegraphics[width=0.35\textwidth]{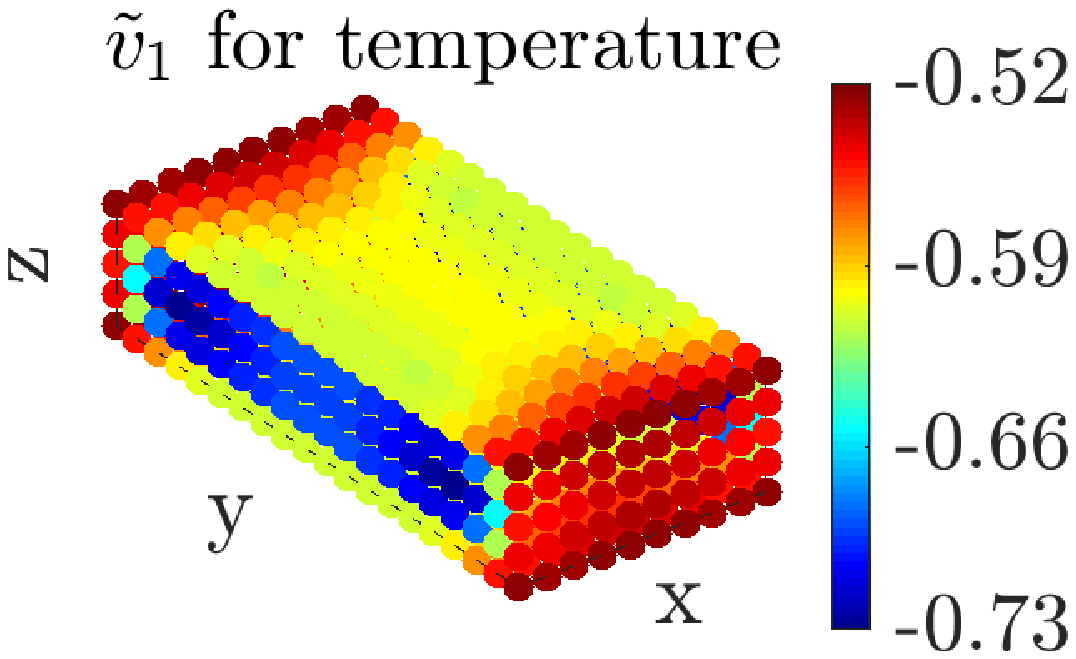}}
	\subfigure{\includegraphics[width=0.35\textwidth]{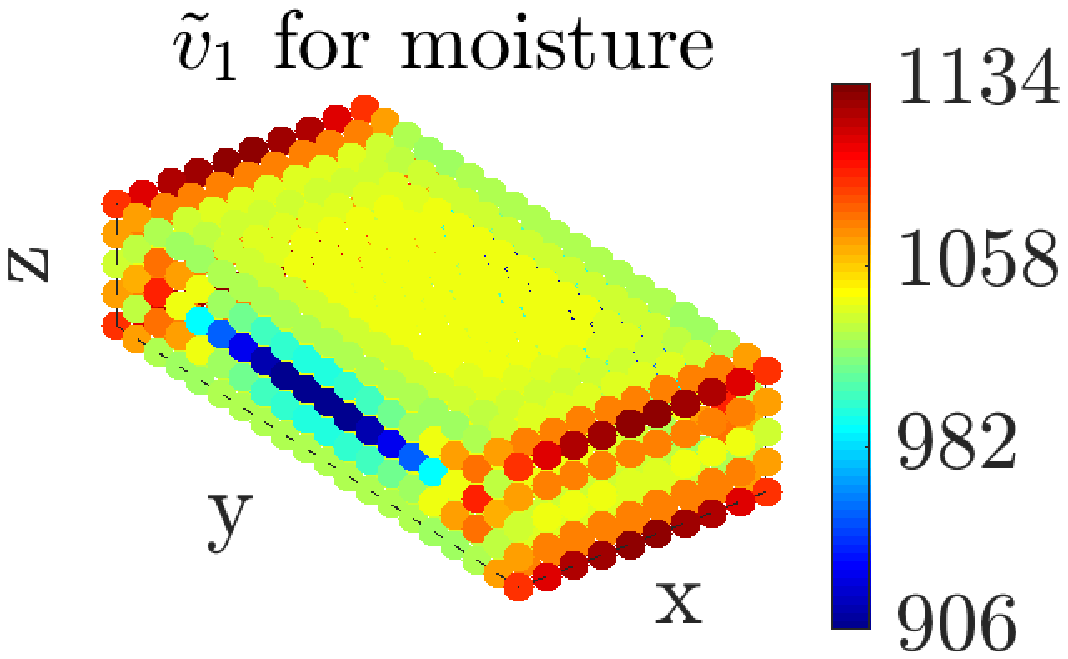}}
	\caption{Approximation of the first eigenvector $\tilde{v}_{1}$ determined with a reduced order model of order $n=10$. The first $1\ldots N $ entries of $\tilde{v}_1$ refer to the temperature (left) and the last $N+1 \ldots 2N$ entries refer to the moisture (right).} 
	\label{fig:MostControllableModes}
\end{figure}

\begin{figure}[b]
	\centering
		\subfigure{\includegraphics[width=0.7\textwidth]{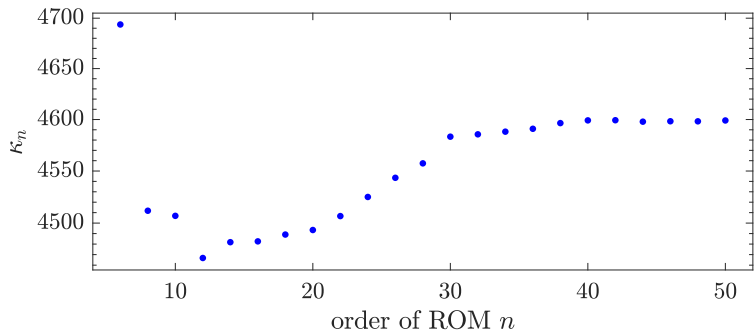}}
		\caption{Observability measure $\kappa_{n}$ for reduced order models of different orders $n=6,8,\ldots,50$ and $n_x= n_T= n/2$. }
	\label{fig:SingularValuesGram}
\end{figure}

Finally, it remains to check how observability varies as a function of the order $n$ of the reduced order models. 
We repeat the observability analysis for the output~\eqref{eqn:Output} and $n= 6, 8, \dots, 50$, where $n= 6$ is the smallest order that resulted in a stable reduced order model and $n= 50$ was chosen, because the observability properties do not change for larger values (see Fig.~\ref{fig:SingularValuesGram}).  

The values of $\kappa$ for the output~\eqref{eqn:Output} are shown in Figure~\ref{fig:SingularValuesGram} as a function of $n$. 
The shown values converge to $\kappa= 4599$. Since all values are within $\pm 3\%$ of the converged value, we conclude the observabilities are acceptable for all shown orders. 
Since $n\ge 10$ was found to be required to obtain the desired approximation accuracy ($E> 99.99\%$, see Section~\ref{subsec:ROMevaluation}), we use $n= 10$ in the remainder of the present work to ensure both sufficient observability and approximation accuracy. 

\subsection{State observer for the drying of wood chips}\label{sec:EKF}


We apply the EKF to the reduced order model \eqref{eqn:PODGalModel} in order to estimate its states \eqref{eqn:PODGalStates} from measurements of the output \eqref{eqn:Output}.
We refer to Appendix B for details on the EKF. 

Let $\hat{c}(t_j)$ refer to the estimated state of~\eqref{eqn:PODGalModel} at time $t_j$, $j=1,\ldots,m$.  
The temperature and moisture estimate, $\hat{T}(y_i,t_j)$ and $\hat{x}(y_i,t_j)$, respectively,  can then be determined by substituting $c(t)$ by $\hat{c}(t_j)$ in \eqref{eqn:SumLinApprox} and \eqref{eqn:SumLinApproxWater}, respectively.
The initial guess for the state and its covariance are denoted by $\hat{c}_0$ and $P_0$, respectively. 


\begin{figure}[h!]
	\centering
	\subfigure{\includegraphics[width=0.35\textwidth]{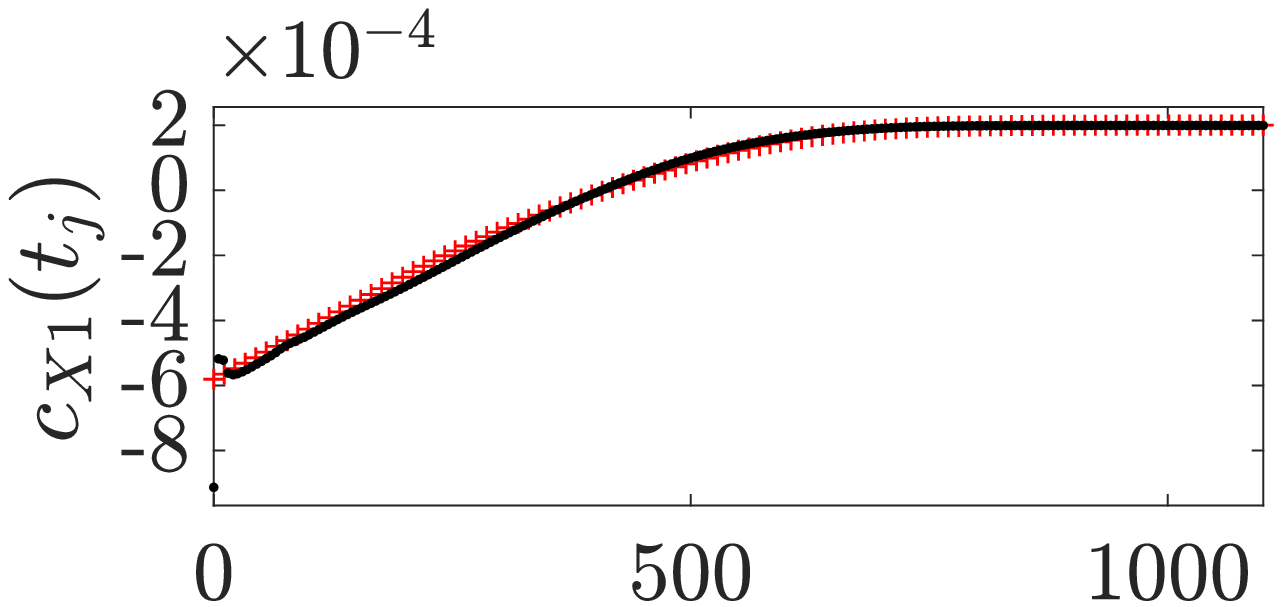}}
	\subfigure{\includegraphics[width=0.35\textwidth]{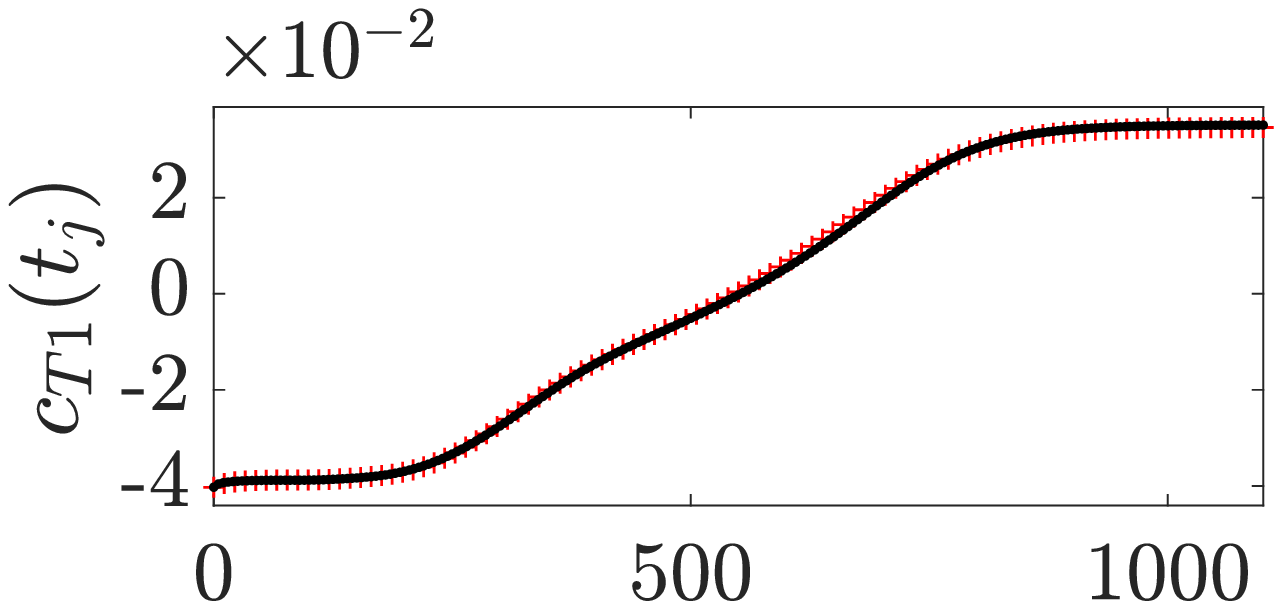}}
	\subfigure{\includegraphics[width=0.35\textwidth]{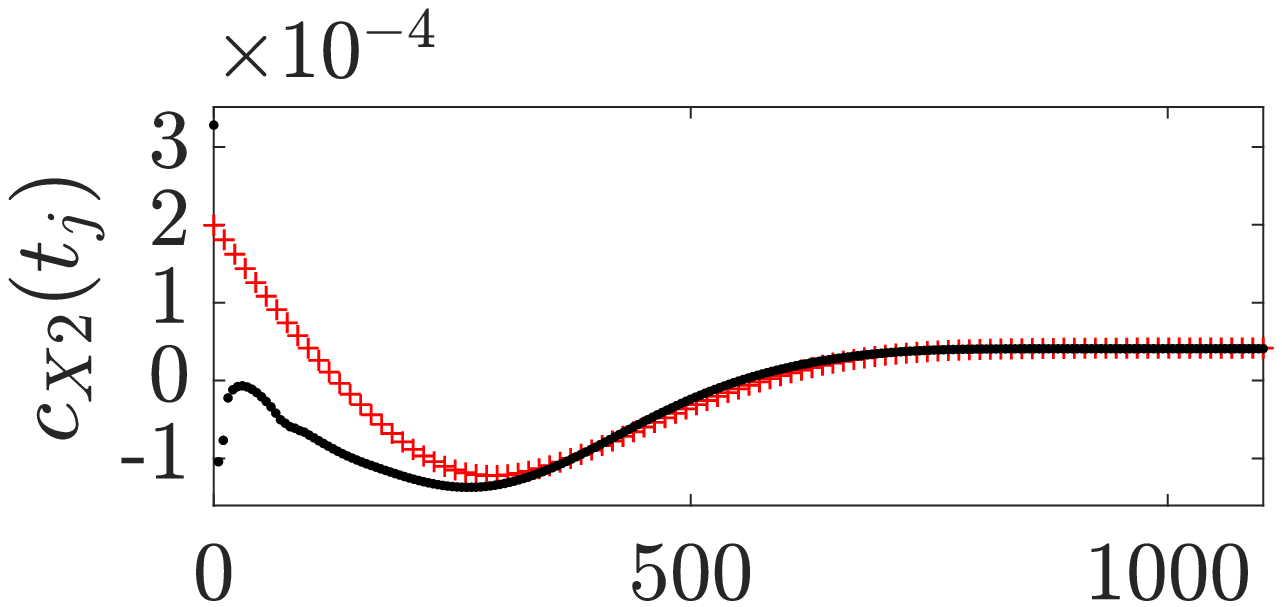}}
	\subfigure{\includegraphics[width=0.35\textwidth]{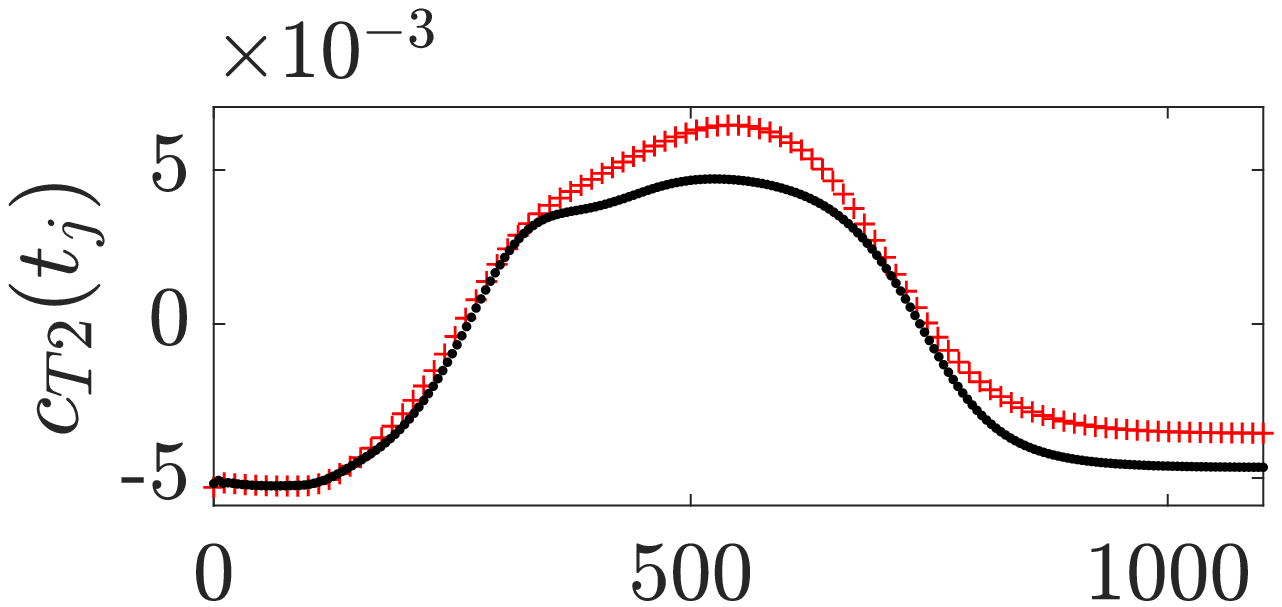}}
	\subfigure{\includegraphics[width=0.35\textwidth]{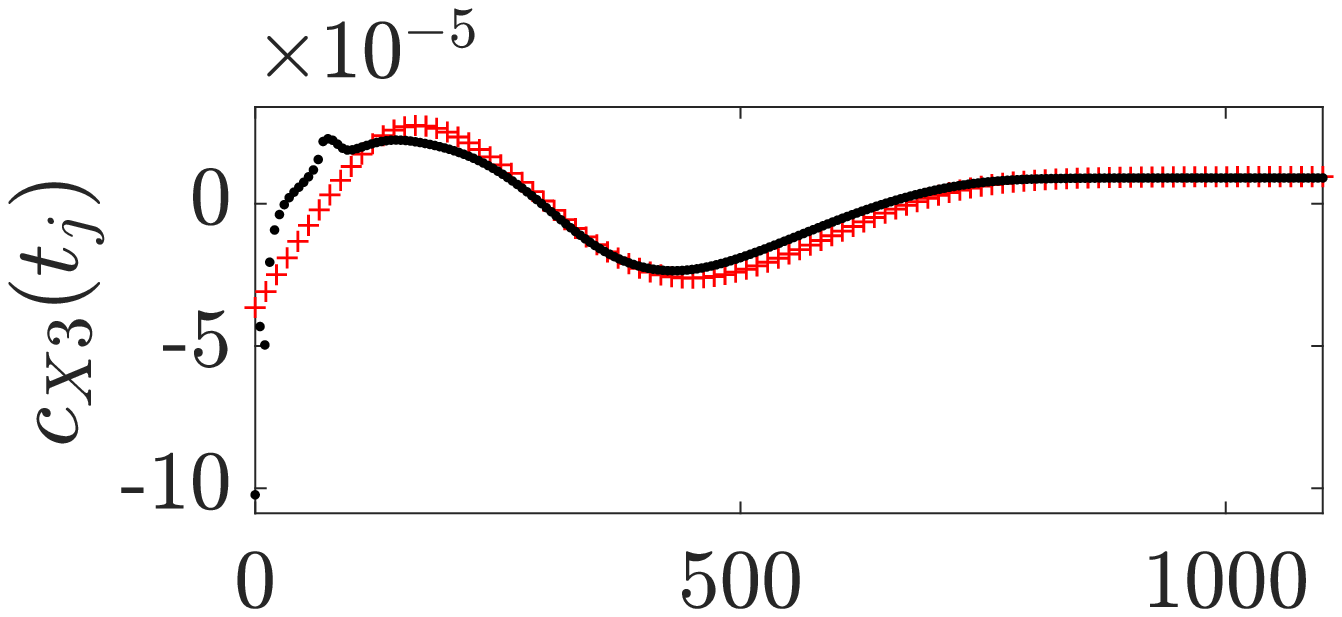}}
	\subfigure{\includegraphics[width=0.35\textwidth]{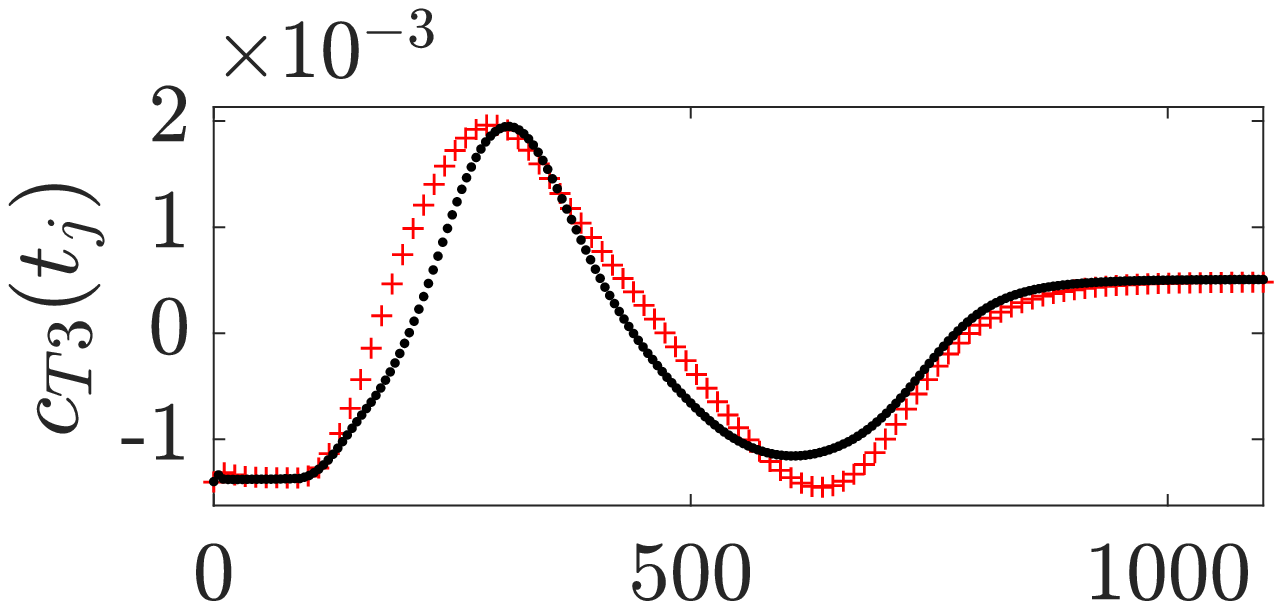}}
	\subfigure{\includegraphics[width=0.35\textwidth]{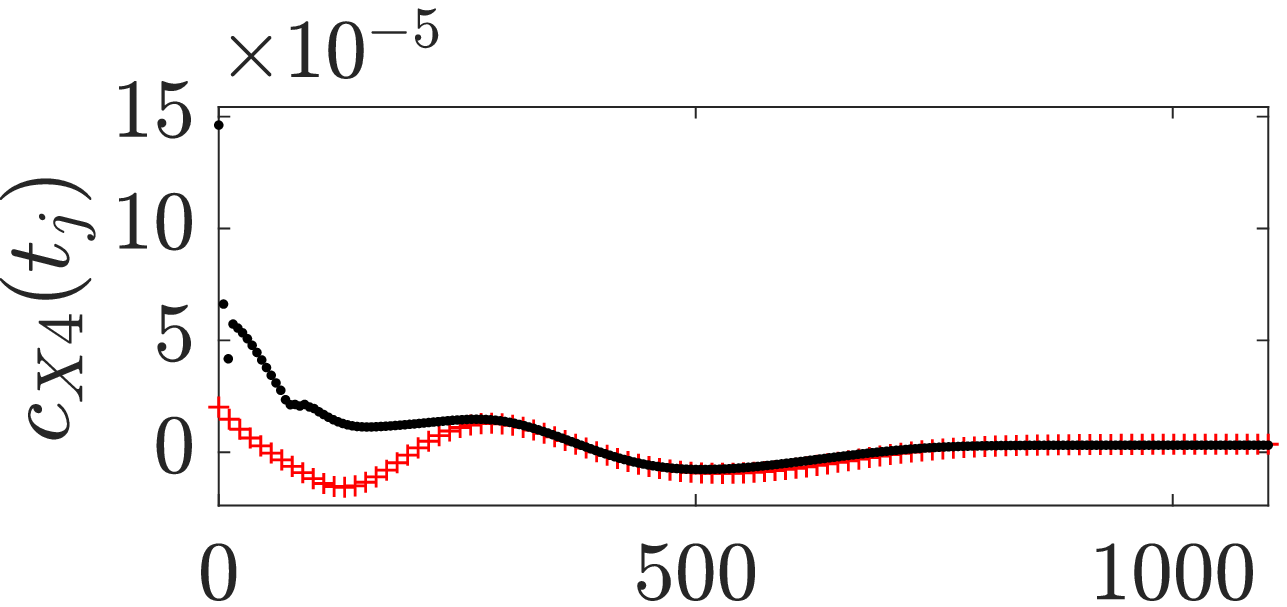}}
	\subfigure{\includegraphics[width=0.35\textwidth]{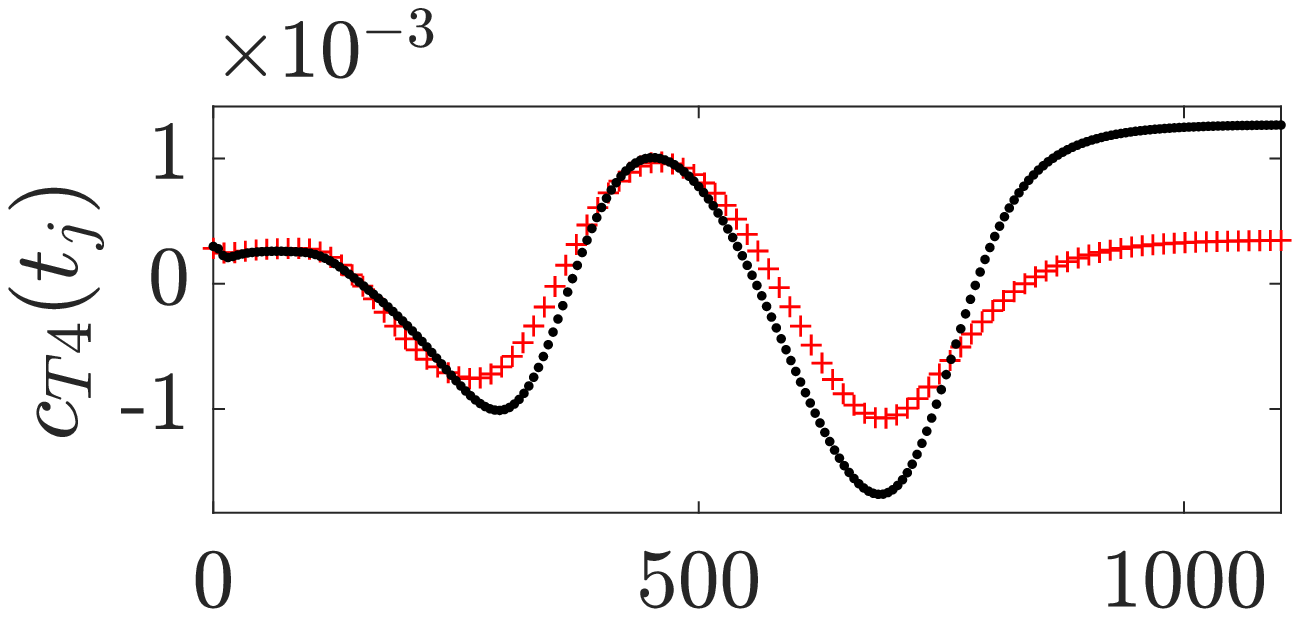}}
	\subfigure{\includegraphics[width=0.35\textwidth]{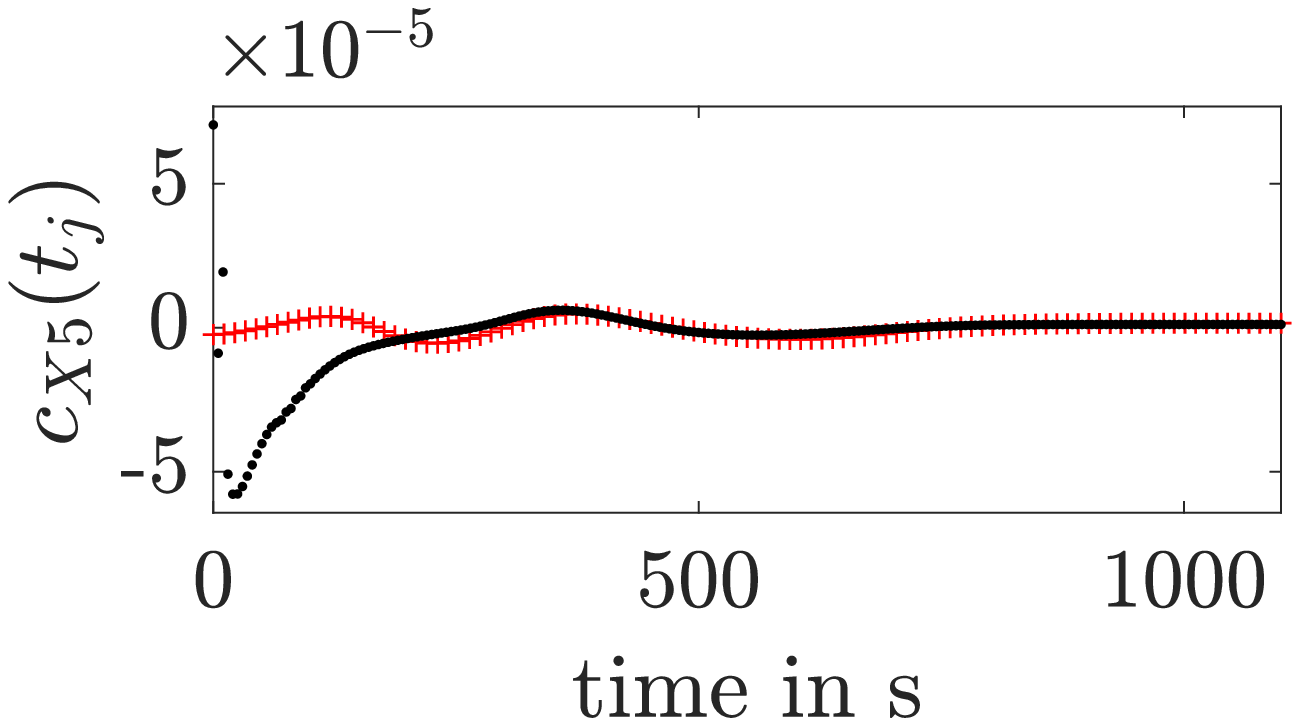}}
	\subfigure{\includegraphics[width=0.35\textwidth]{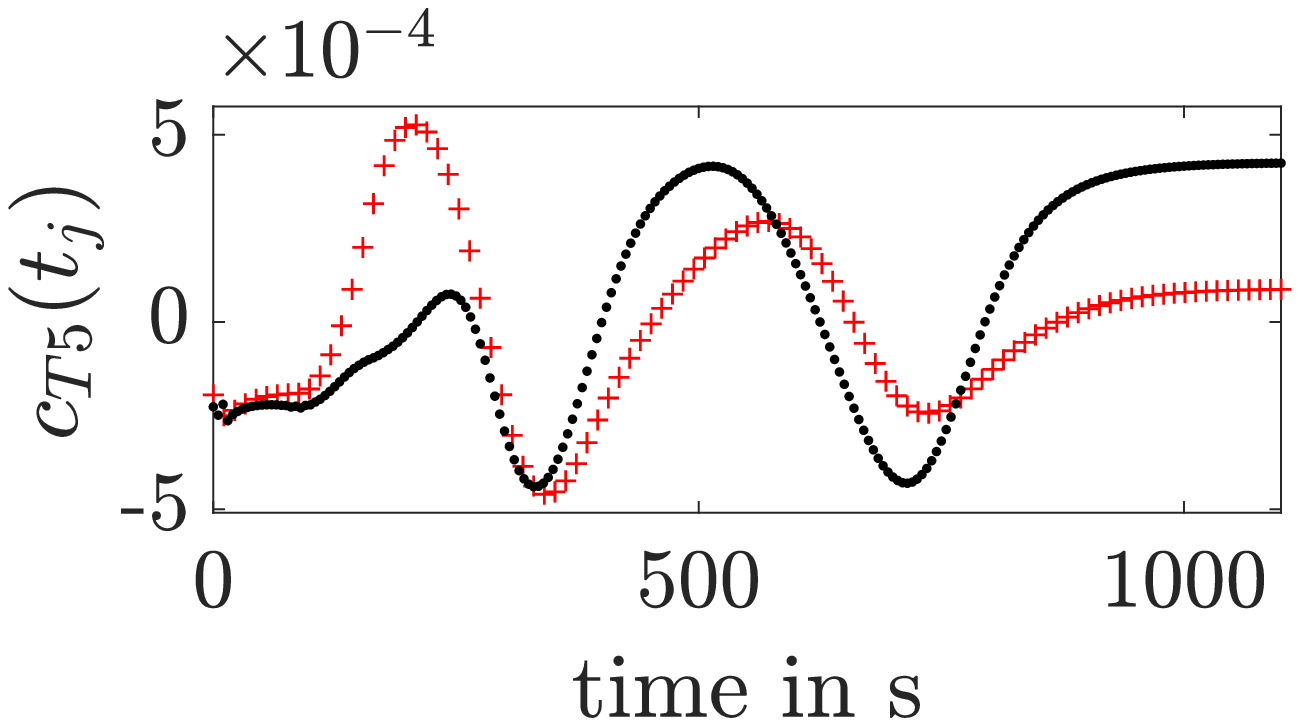}}
	\caption{Kalman filter estimation (black) for moisture (left) and temperature (right) related coefficients. POD coefficients (red) for comparison.}
	\label{fig:EKFCoeffInitMeasurement}
\end{figure}

\begin{figure}[t]
	\centering
	\subfigure{\includegraphics[width=0.7\textwidth]{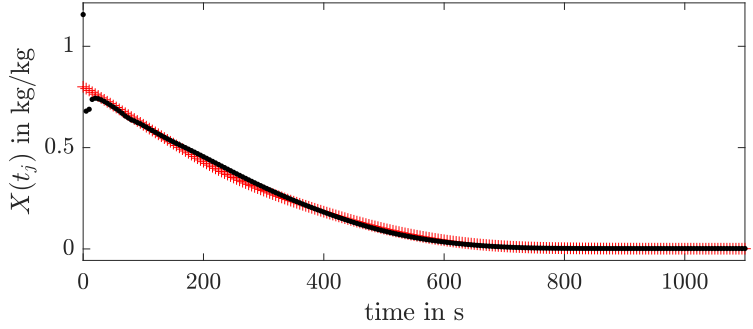}}
	\caption{Total moisture estimate (black) and true values (red). The EKF is initialized with temperature measurements and an initial guess of $x(y_i,t=0)=\unit[1]{kg/kg}$ for the moisture.}
	\label{fig:EKFEstTotMoistureInitMeasurement}
\end{figure}

\begin{figure}[t]
	\centering
	\subfigure{\includegraphics[width=0.7\textwidth]{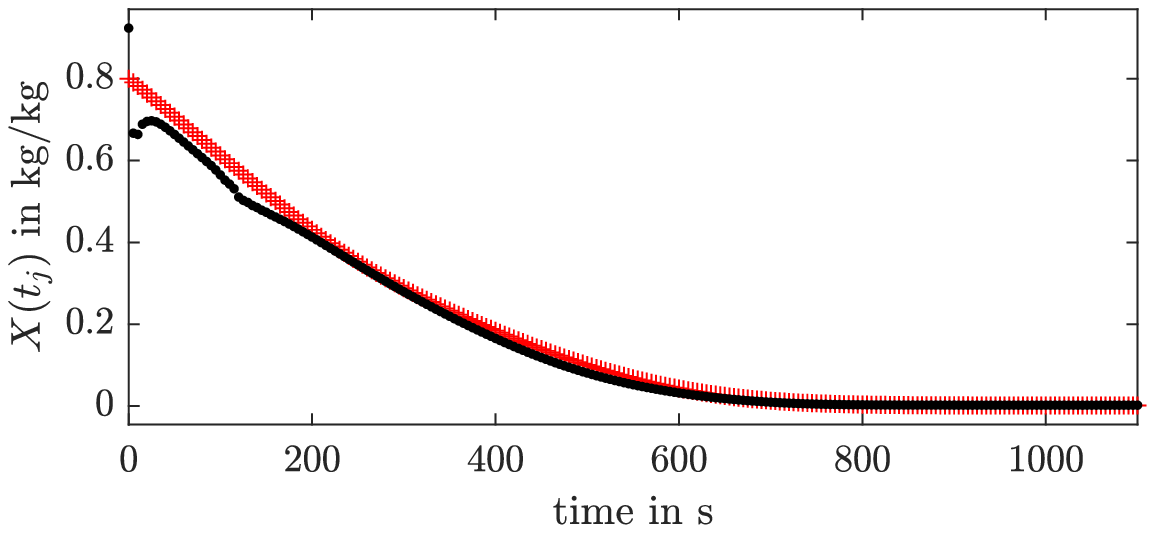}}
	\caption{The total moisture $X(t_j)$ (red) is compared to its EKF estimate (black). The EKF is initialized with temperature measurements and an initial guess of $x(y_i,t=0)=\unit[0.6]{kg/kg}$ for the moisture.}
	\label{fig:EKFEstimatedTotalMoisture}
\end{figure}

Measurements for \eqref{eqn:Output} are taken every $\unit[5]{s}$ at times $t_j$, $j=1,\ldots,m$ until a steady state is reached after $\unit[1100]{s}$ ($m=220$). The computation time required for the EKF, i.e., the equations stated in Appendix B including the reduced order model, is well below $\unit[5]{s}$.\footnote[2]{Specifically, $\unit[0.038]{s}$ are required with a matlab implementation on a desktop PC with an Intel i7-6700 CPU running at 3.4 GHz.}

It is an obvious choice to use 
the measurement \eqref{eqn:Output} at $t_0=0$ as an initial guess for the temperature states. 
The desired states are determined by $\hat{c}_{0,T,k}= \textstyle\sum_{i=1}^{N} \big( w(t=0)- \bar{T}(y_i)\big) \, \varphi_{T,k}(y_i) \Delta V$ according to \eqref{eqn:TimeCoeff}
for a particle with an assumed homogeneous temperature distribution $T(y_i,t=0)=w(t=0)$ at all $y_i$, $i=1,\ldots,N$, 
The initial state for the moisture cannot be based on measurements, since we only measure temperatures. 
We overestimate and underestimate the true initial moisture by $25\;\%$ and show the estimation using the correct initial moisture $x(y_i,t=0)=0.8$ for comparison.

An overestimation of the initial moisture by $25\%$ yields $x(y_i,t=0)=1$ for all $y_i$ and results in $\hat{c}_{0,x,k}= \textstyle\sum_{i=1}^{N} \big( 1 - \bar{x}(y_i)\big) \, \varphi_{x,k}(y_i) \Delta V$ for the initial state for the moisture.
We choose $Q=I_n$, $R=1$, $P_0= 2\times 10^2\cdot I_n$ for the process, measurement noise and error covariance matrices, respectively, where $I_n$ refers to the $n\times n$ unity matrix.

The resulting state estimate $\hat{c}(t_j)$ is shown in Figure~\ref{fig:EKFCoeffInitMeasurement}. 
The coefficients of \eqref{eqn:SumLinApprox} and \eqref{eqn:SumLinApproxWater} obtained by POD are shown for comparison. We observe good agreement for the most important modes. The extended Kalman filter needs about $\unit[200]{s}$ to converge to the true estimate for some moisture related states, while all temperature states converge practically immediately. This is in agreement with the observations of Section~\ref{subsec:OBSVDryingProcess}, where we found weaker observability for the moisture distribution. After $\unit[200]{s}$, the estimated states match well for both temperature and moisture. Only some deviations can be observed for higher order modes.

When combined with the POD basis, the states
$\hat{c}(t_j)$ provide an estimate of the temperature and moisture distribution on the whole domain $\Omega$. 
We use the  normalized root mean square error \eqref{eqn:MeanRelativeError} to evaluate the quality of the estimation with the reduced order model, where $c(t_j)$ is substituted by $\hat{c}(t_j)$ in \eqref{eqn:AbsError}. This results in $\varepsilon_T(y_i,t_j)=1.7\%$  for the temperature and $\varepsilon_x(y_i,t_j)=8.4\%$ for the moisture estimation. 

The resulting estimate for the total moisture~\eqref{eqn:TotalMoisture} is shown in Figure~\ref{fig:EKFEstTotMoistureInitMeasurement}. 
Deviations occur only during the first $\unit[10]{s}$ and all values match well for the remaining time span. The 
normalized root mean square error amounts to $\varepsilon_X =3.5\%$. The error induced by the total moisture $X(t)$ is smaller than the error induced by the spatially distributed moisture $x(y,t)$ due to the averaging nature of \eqref{eqn:TotalMoisture}.
We conclude the errors are in an acceptable range and the presented method provides a good estimation of the inner particle moisture during the drying process. 

We now consider the underestimation of the initial moisture. In this case $x(y_i,t=0)=0.6$ for all $y_i$, which results in $\hat{c}_{0,x,k}= \textstyle\sum_{i=1}^{N} \big( 0.6 - \bar{x}(y_i)\big) \, \varphi_{x,k}(y_i) \Delta V$ for the initial state for the moisture. We choose $P_0=83.5 \cdot I_n$ for fast convergence. We briefly note that the Kalman filter also converges for other $P_0$, but generally $P_0$ is chosen larger when we have only little confidence in the initial state guesses $\hat{c}_0$.
The total moisture estimate is depicted in Figure~\ref{fig:EKFEstimatedTotalMoisture}, which shows that the EKF converges after about $\unit[200]{s}$. 
The normalized root mean square errors are about as large as for the overestimated initial moisture. Specifically, we obtain $\varepsilon_x = 8.2\%$ and $\varepsilon_X =3.25\%$ for the errors of the moisture and total moisture, respectively, and $\varepsilon_T =2.1\%$ for the temperature error.
We claim without giving details that shorter sampling time or other outputs than \eqref{eqn:Output}, e.g., a higher dimensional output that considers the spatially resolved surface temperature, do not yield better results. The increased errors therefore can be accounted to the worse initial guess. 


The last case uses the true initial temperature and moisture values as an initial guess. The $\hat{c}_0$ are taken from the snapshots of the POD decompositions \eqref{eqn:SumLinApprox} and \eqref{eqn:SumLinApproxWater} for $t=\unit[0]{s}$. For $P_0=0_n$, the estimated total moisture is presented in Figure~\ref{fig:EKFEstimatedTotalMoisture3}. We observe good agreement of the estimated and true total moisture. Only little deviations are observable during the first $\unit[100]{s}$.
The normalized root mean square error amounts to $\varepsilon_X =2.4\%$ for the total moisture and $\varepsilon_T =1.8\%$ for the temperature.
The error for the moisture is $\varepsilon_x =3.1\%$ and, therefore, nearly one third of the error in the first case with measurements used for the initial state guess.

\begin{figure}[t]
	\centering
	\subfigure{\includegraphics[width=0.7\textwidth]{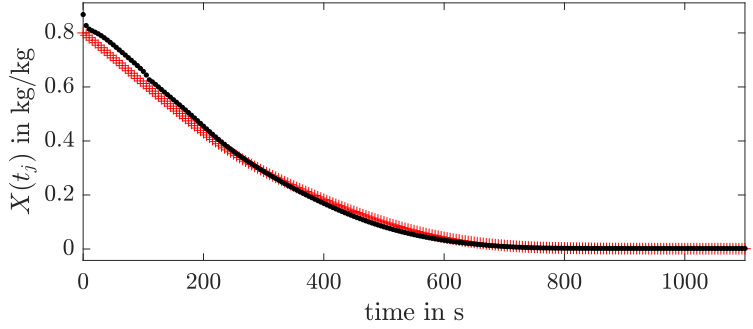}}
	\caption{Total moisture estimate (black) and true values (red). POD coefficients were used as initial guess in the EKF.}
	\label{fig:EKFEstimatedTotalMoisture3}
\end{figure}

\section{Conclusion}\label{sec:Outlook}
We showed the moisture content of wood chips can be estimated with reduced order models. 
Our observability analysis and computational experiments reveal the moisture of the particle can be determined from surface temperature measurements alone. 
In particular, no moisture measurements are required, which would render the approach impractical for an industrial use. 
Because calculations with PDEs can be avoided after the model reduction, 
the extended Kalman filter is computationally lean and therefore a viable candidate for a practical implementation.

%
%

\appendix
\section*{Appendix A}\label{appx:A}
The volume integral of the inner product in \eqref{eqn:PDEProjection2} is transformed into a surface integral with the help of Gauss's theorem. The boundary condition \eqref{eqn:PDEBCT} then appears explicitly in \eqref{eqn:PDEProjection2} (see \cite{Berner2017} for details). We obtain 
\begin{align}
&f_{\text{ROM,T},l} \approx -\textstyle\int_\Omega  \Big(\lambda \nabla \big( \bar{T}+\textstyle\sum_{k=1}^{n_T}\varphi_{T,k}c_{T,k}\big)\Big) \cdot \nabla \varphi_{T,l} \diff V + \nonumber \\
& \textstyle\int_{\partial \Omega} \varphi_{T,l} \, s^{-1}\, \Big( J_T \big(\bar{x}+\textstyle\sum_{k=1}^{n_x}\varphi_{x,k}c_{x,k}, \, \bar{T}+\textstyle\sum_{k=1}^{n_T}\varphi_{T,k}c_{T,k} \big) +\nonumber\\
&\Delta h_\mathrm{v} \varrho_\mathrm{d} J_{T,\mathrm{v}} \big(\bar{x}+\textstyle\sum_{k=1}^{n_x}\varphi_{x,k}c_{x,k}, \, \bar{T}+\textstyle\sum_{k=1}^{n_T}\varphi_{T,k}c_{T,k} \big)  \Big) \diff S  \label{eqn:ROMGaussT}
\end{align}
$l=1,\ldots,n_T$, for the temperature and
\begin{align}
&f_{\text{ROM,x},l}\approx -\textstyle\int_\Omega  \Big(\delta \nabla \big( \bar{x}+\textstyle\sum_{k=1}^{n_x}\varphi_{x,k}c_{x,k}\big)\Big) \cdot \nabla \varphi_{x,l} \diff V + \nonumber\\
& \textstyle\int_{\partial \Omega} \varphi_{x,l} \, J_{x} \big(\bar{x}+\textstyle\sum_{k=1}^{n_x}\varphi_{x,k}c_{x,k}, \, \bar{T}+\textstyle\sum_{k=1}^{n_T}\varphi_{T,k}c_{T,k} \big) \diff S   \label{eqn:ROMGaussX}
\end{align}
$l=1,\ldots,n_x$, for the moisture. The system of $n = n_T+n_x$ ODEs constitutes the reduced model. Note that all equations \eqref{eqn:ROMGaussT} and \eqref{eqn:ROMGaussX} are coupled since both, temperature and moisture related states, i.e., $c_{T,k}$ and $c_{x,k}$, respectively, appear in all ODEs. Furthermore, \eqref{eqn:ROMGaussT} and \eqref{eqn:ROMGaussX} are nonlinear due to the nonlinearity of the boundary conditions \eqref{eqn:PDEBC}.

\section*{Appendix B}\label{appx:B}
We initialize the state and covariance estimate with $\hat{c}_{0|0}=\hat{c}_0$ and $P_{0|0}=P_0$,
where we denote the estimation of a variable $a(t_i)$ estimated at time $t_j$ by $a_{i|j}$. 
For every time step $k=1,\ldots,m$ we perform the following steps. The prediction step is based on the state and covariance estimates
\begin{align*}
\hat{c} (t_{k-1}) &= \hat{c}_{k-1|k-1} \\
P(t_{k-1}) &= P_{k-1|k-1},
\end{align*}
respectively, of the previous time step at $t_{k-1}$. Then
\begin{align*}
\dot{\hat{c}} (t)&= f_\text{ROM}\big( \hat{c}(t) \big)\\
\dot{P}(t) &= F(t)P(t) + P(t)F(t)^\T+ Q(t)
\end{align*}
are solved to predict the state and covariance estimates 
\begin{align*}
\hat{c}_{k|k-1} &= \hat{c}(t_k) \\
P_{k|k-1} &= P(t_k),
\end{align*}
respectively, at the current time step $t_k$. The Jacobian of the system reads $F(t) = \frac{\partial f_\text{ROM}}{\partial c(t)} \big \vert_{\hat{c}(t)}$. The state and covariance are updated by
\begin{align*}
\hat{c}_{k|k} &= \hat{c}_{k|k-1} + K_k \Big(w_k - h\big( \Phi\hat{c}_{k|k-1} +\bar{z}\big)\Big)\\
P_{k|k} &= (I - K_k H_k) P_{k|k-1},
\end{align*}
where the Kalman gain and the Jacobian of the output function read
\begin{align*}
K_k &= P_{k|k-1}H_k^\T\big(H_k P_{k|k-1} H_k^\T + R_k \big)^{-1}\\
 H_k &= \frac{\partial h\big( \Phi c(t) +\bar{z}\big)}{\partial c(t) } \bigg \vert _{\hat{c}_{k|k-1}},
\end{align*}
respectively. The time series $\hat{c}(t_k) = \hat{c}_{k|k}$, $k=1,\ldots,m$, then yields the desired state estimate of \eqref{eqn:PODGalModel}.

 \bibliographystyle{plain} 
\bibliography{BernerO2020}         

\begin{thebibliography}{10}

\bibitem{Bengtsson2008}
P.~Bengtsson.
\newblock Experimental analysis of low-temperature bed drying of wooden biomass
  particles.
\newblock {\em Drying Technology}, 26(5):602--610, 2008.

\bibitem{Berner2019}
M.~O. Berner, V.~Scherer, and M.~M\"{o}nnigmann.
\newblock Controllability analysis and optimal control of biomass drying with
  reduced order models.
\newblock {\em Journal of Process Control (accepted)}, 2019.
\newblock , arXiv:1911.04688v1.

\bibitem{Berner2017}
M.~O. Berner, F.~Sudbrock, V.~Scherer, and M.~M\"onnigmann.
\newblock \uppercase{POD} and \uppercase{G}alerkin-based reduction of a wood
  chip drying model.
\newblock {\em IFAC Papers\-On\-Line}, 50:6619--6623, 2017.

\bibitem{Chen1999}
C.-T. Chen.
\newblock {\em Linear System Theory and Design}.
\newblock Oxford University Press, 3rd edition, 1999.

\bibitem{Cordier2008a}
L.~Cordier and M.~Bergmann.
\newblock Proper orthogonal decomposition: \uppercase{A}n overview.
\newblock In P.~Millan and M.~L. Riethmuller, editors, {\em Post-Processing of
  numerical and experimental data}, pages 1--45. Von Karman Institute for Fluid
  Dynamics, 2008.

\bibitem{Cordier2008b}
L.~Cordier and M.~Bergmann.
\newblock Two typical applications of \uppercase{POD}: Coherent structures
  education and reduced order modelling.
\newblock In P.~Millan and M.~L. Riethmuller, editors, {\em Post-Processing of
  numerical and experimental data}, pages 1--60. Von Karman Institute for Fluid
  Dynamics, 2008.

\bibitem{Deane1991}
A.~E. Deane, I.~G. Kevrekidis, G.~E. Karniadakis, and S.~A. Orszag.
\newblock Low-dimensional models for complex geometry flows: Application to
  grooved channels and circular cylinders.
\newblock {\em Physics of Fluids}, 3:2337--2354, 1991.

\bibitem{Delrattre2004}
C.~Delattre, D.~Dochain, and J.~Winkin.
\newblock Observability analysis of nonlinear tubular (bio)reactor models:
  \uppercase{A} case study.
\newblock {\em Journal of Process Control}, 14:661--669, 2004.

\bibitem{Eymard2000}
R.~Eymard, T.~Gallou\"{e}t, and R.~Herbin.
\newblock Finite volume methods.
\newblock In P.~G. Ciarlet and J.~L. Lions, editors, {\em Handbook of Numerical
  Analysis}, volume~7, pages 713--1018. Elsevier, 2000.

\bibitem{Fletcher1984}
C.~A.~J. Fletcher.
\newblock {\em Computational Galerkin Methods}.
\newblock Springer Series in Computational Physics. Springer, 1984.

\bibitem{Hahn2002}
J.~Hahn and T.~F. Edgar.
\newblock An improved method for nonlinear model reduction using balancing of
  empirical \uppercase{G}ramians.
\newblock {\em Computers \& Chemical Engineering}, 26:1379--1397, 2002.

\bibitem{Hahn2003}
J.~Hahn, T.~F. Edgar, and W.~Marquardt.
\newblock Controllability and observability covariance matrices for the
  analysis and order reduction of stable nonlinear systems.
\newblock {\em Journal of Process Control}, 13:115--127, 2003.

\bibitem{Hoepffner2005}
J.~H{\oe}pffner, M.~Chevalier, T.~Bewley, and D.S. Henningson.
\newblock State estimation in wall-bounded flow systems. \uppercase{P}art 1.
  \uppercase{P}erturbed laminar flows.
\newblock {\em Journal of Fluid Mechanics}, 534:263--294, 2005.

\bibitem{Jansen2017}
J.~D. Jansen and L.~J. Durlofsky.
\newblock Use of reduced-order models in well control optimization.
\newblock {\em Optimization and Engineering}, 18:105--132, 2017.

\bibitem{John2010}
T.~John, M.~Guay, N.~Hariharan, and S.~Naranayan.
\newblock \uppercase{POD}-based observer for estimation in
  \uppercase{N}avier-–\uppercase{S}tokes flow.
\newblock {\em Computers and Chemical Engineering}, 34:965--975, 2010.

\bibitem{Lall1999}
S.~Lall, J.~E. Marsden, and S.~Glava\v{s}ki.
\newblock Empirical model reduction of controlled nonlinear systems.
\newblock {\em IFAC Proceedings Volumes}, 32:2598--2603, 1999.

\bibitem{Leon2002}
L.~Le\'{o}n and E.~Zuazua.
\newblock Boundary controllability of the finite-difference space
  semi-discretizations of the beam equation.
\newblock {\em ESAIM: Control, Optimisation and Calculus of Variations},
  8:827--862, 2002.

\bibitem{Moore1981}
B.~C. Moore.
\newblock Principal component analysis in linear systems: Controllability,
  observability, and model reduction.
\newblock {\em IEEE Transactions on Automatic Control}, 26:17--32, 1981.

\bibitem{Moreton2002}
S.~Moreton.
\newblock Silica gel impregnated with iron(\uppercase{III}) salts:
  \uppercase{A} safe humidity indicator.
\newblock {\em Material Research Innovations}, 5(5):226--229, 2002.

\bibitem{Moukalled2015}
F.~Moukalled, L.~Mangani, and M.~Darwish.
\newblock {\em The finite volume method in computational fluid dynamics}.
\newblock Fluid Mechanics and Its Applications. Springer, 2015.

\bibitem{Rickelt2013}
S.~Rickelt, F.~Sudbrock, S.~Wirtz, and V.~Scherer.
\newblock Coupled \uppercase{DEM}/\uppercase{CFD} simulation of heat transfer
  in a generic grate system agitated by bars.
\newblock {\em Powder Technology}, 249:360--372, 2013.

\bibitem{RuizLopez2011}
I.I. Ruiz-Lopez, H.~Ruiz-Espinosa, M.L. Luna-Guevara, and M.A. Garcia-Alvarado.
\newblock Modeling and simulation of heat and mass transfer during drying of
  solids with hemispherical shell geometry.
\newblock {\em Computers \& Chemical Engineering}, 35(2):191--199, 2011.

\bibitem{Scherer2016}
V.~Scherer, M.~M\"{o}nnigmann, M.~O. Berner, and F.~Sudbrock.
\newblock Coupled \uppercase{DEM}--\uppercase{CFD} simulation of drying wood
  chips in a rotary drum -- baffle design and model reduction.
\newblock {\em Fuel}, 184:896--904, 2016.

\bibitem{Singh2005}
A.~K. Singh and Juergen Hahn.
\newblock Determining optimal sensor locations for state and parameter
  estimation for stable nonlinear systems.
\newblock {\em Industrial and Engineering Chemistry Research}, 44:5645--5659,
  2005.

\bibitem{Sirovich1987}
L.~Sirovich.
\newblock Turbulence and the dynamics of coherent structures, \uppercase{P}art
  \uppercase{I}-\uppercase{III}.
\newblock {\em Quarterly of Applied Mathematics}, 45(3):561--590, 1987.

\bibitem{Sudbrock2014}
F.~Sudbrock.
\newblock {\em \uppercase{DEM}/\uppercase{CFD} analysis for the convective
  drying of agitated beds}.
\newblock Ruhr-Universit\"at Bochum, PhD thesis, Shaker, (in German), 2015.

\bibitem{Sudbrock2015}
F.~Sudbrock, H.~Kruggel-Emden, S.~Wirtz, and V.~Scherer.
\newblock Convective drying of agitated silica gel and beech wood particle beds
  -- \uppercase{E}xperiments and transient \uppercase{DEM}-\uppercase{CFD}
  simulations.
\newblock {\em Drying Technology}, 33(15-16):1808--1820, 2015.

\bibitem{DeTemmerman2009}
J.~De Temmerman, P.~Dufour, B.~Nicolai, and H.~Ramon.
\newblock {MPC} as control strategy for pasta drying processes.
\newblock {\em Computers \& Chemical Engineering}, 33(1):50--57, 2009.

\bibitem{Tsuji2010}
T.~Tsuji, T.~Miyauchi, S.~Oh, and T.~Tanaka.
\newblock Simultaneous measurement of particle motion and temperature in
  two-dimensional fluidized bed with heat transfer.
\newblock {\em \uppercase{KONA} \uppercase{P}owder and \uppercase{P}article
  \uppercase{J}ournal}, (28):167--179, 2010.

\bibitem{Zuazua2007}
Enrique Zuazua.
\newblock Controllability and observability of partial differential equations:
  Some results and open problems.
\newblock In C.~M. Dafermos and E.~Feireisl, editors, {\em Handbook of
  Differential Equations: Evolutionary Equations}, chapter~7, pages 527--621.
  Elsevier, 2007.

\end{thebibliography}

\end{document}